\documentclass[11pt,a4paper]{article}

\usepackage{cmap}
\usepackage[T1]{fontenc}

\usepackage{tocloft}

\usepackage{amsmath,amsthm,amssymb,latexsym,graphicx,mathrsfs}
\usepackage{secdot}
\sectiondot{subsection}
\usepackage{a4wide}
\usepackage{float}

\usepackage{mathtools}
\mathtoolsset{showonlyrefs}
\allowdisplaybreaks

\usepackage{marginnote}

\usepackage{xcolor}
\usepackage{url}
\usepackage{hyperref}
\definecolor{ForestGreen}{rgb}{0.1,0.6,0.05}
\definecolor{EgyptBlue}{rgb}{0.063,0.1,0.6}
\definecolor{RipeOlive}{HTML}{556B2F}
\hypersetup{
	colorlinks=true,
	linkcolor=EgyptBlue,         
	citecolor=ForestGreen,
	urlcolor=RipeOlive
}

\usepackage[hyperpageref]{backref}
\usepackage{epstopdf}
\newcounter{dummy}
\usepackage{enumitem}
\makeatletter
\newcommand\myitem[1][]{\item[#1]\refstepcounter{dummy}\def\@currentlabel{#1}}
\makeatother

\DeclareRobustCommand\Plef{\mathcal{P}}

\newtheorem{theorem}{Theorem}
\newtheorem{proposition}[theorem]{Proposition}
\newtheorem{lemma}[theorem]{Lemma}
\newtheorem{corollary}[theorem]{Corollary}

\theoremstyle{definition}

\newtheorem{remark}[theorem]{Remark}

\numberwithin{equation}{section}
\numberwithin{theorem}{section}

\numberwithin{equation}{section}
\numberwithin{theorem}{section}

\newenvironment{proof*}[1]{\begin{trivlist}\item[\hskip%
		\labelsep{{\bf Proof of \/{\rm\bf #1.}}~}]\rm}%
	{\hfill\qed\rm\end{trivlist}}

\newcommand{\W}{W_0^{1,p}}

\newcommand{\intO}{\int_\Omega}

\newcommand{\E}{E_{\lambda,\eta}} 
\newcommand{\En}{E_{\lambda_n,\eta}}

\title{
	\vspace*{-1cm}
	On the antimaximum principle for the $p$-Laplacian\\ and its sublinear perturbations
} 
\author{Vladimir Bobkov ~\&~ Mieko Tanaka \\}
\date{}

\AtEndDocument{%
	\par
	\bigskip
	\bigskip
	\begin{tabular}{@{}l@{}}%
		(V.~Bobkov)\\[0.2em]
		\textsc{Institute of Mathematics, Ufa Federal Research Centre, RAS}\\
		\textsc{Chernyshevsky str. 112, Ufa 450008, Russia}\\[0.3em]
		\textit{E-mail address}: \texttt{bobkov@matem.anrb.ru}\\[1.5em]
		
		(M.~Tanaka)\\[0.2em]
		\textsc{Department of  Mathematics, Tokyo University of Science}\\
		\textsc{Kagurazaka 1-3, Shinjyuku-ku, Tokyo 162-8601, Japan}\\[0.3em]
		\textit{E-mail address}: \texttt{miekotanaka@rs.tus.ac.jp}
\end{tabular}}

\begin{document}
	\maketitle
	
	\begin{abstract}
		We investigate qualitative properties of weak solutions of the Dirichlet problem for the equation $-\Delta_p u = \lambda \,m(x)|u|^{p-2}u+ \eta \,a(x)|u|^{q-2}u+ f(x)$ in a bounded domain $\Omega \subset \mathbb{R}^N$, where $q<p$.
		Under certain regularity and qualitative assumptions on the weights $m, a$ and the source function $f$, we identify ranges of parameters $\lambda$ and $\eta$ for which solutions satisfy maximum and antimaximum principles in weak and strong forms.
		Some of our results, especially on the validity of the antimaximum principle under low regularity assumptions, are new for the unperturbed problem with $\eta=0$, and among them there are results providing new information even in the linear case $p=2$.  
		In particular, we show that for \textit{any} $p>1$ solutions of the unperturbed problem satisfy the antimaximum principle in a right neighborhood of the first eigenvalue of the $p$-Laplacian provided $m,f \in L^\gamma(\Omega)$ with $\gamma>N$. 
		For completeness, we also investigate the existence of solutions.
		
		
	\end{abstract}
	
	\begin{quote}	
		\tableofcontents	
		\addtocontents{toc}{\vspace*{-2ex}}
	\end{quote}

	\section{Introduction}\label{sec:intro}
	
	In the present work, we study how the inclusion of a model indefinite subhomogeneous perturbation into the Fredholm problem for the $p$-Laplacian affects qualitative properties of (weak) solutions such as their obedience to the maximum and antimaximum principles. 
	More precisely, we investigate
	the boundary value problem
	\begin{equation}\label{eq:Psub}
		\tag{$\Plef$}
		\left\{
		\begin{aligned}
			-\Delta_p u &= \lambda \,m(x) |u|^{p-2}u+ \eta \,a(x)|u|^{q-2}u+ f(x) 
			&&\text{in}\ \Omega, \\
			u&=0 &&\text{on}\ \partial \Omega,
		\end{aligned}
		\right.
	\end{equation}
	where the exponents $p,q$ satisfy $1<q<p<\infty$, $\lambda \in\mathbb{R}$ plays a role of the spectral parameter, and the parameter $\eta\in\mathbb{R}$ controls the influence of the subhomogeneous perturbation $a|u|^{q-2}u$. 
	Occasionally, when no ambiguity occurs, we will refer to \eqref{eq:Psub} as  {\renewcommand{\Plef}{\mathcal{P};\lambda,\eta}\eqref{eq:Psub}} or {\renewcommand{\Plef}{\mathcal{P};\lambda,\eta,f}\eqref{eq:Psub}}, in order to reflect the dependence of the problem on the corresponding quantities.
	
	We always assume, by default, that $\Omega \subset \mathbb{R}^N$ is a bounded domain in $\mathbb{R}^N$, $N \geq 1$. For some results, the following more restrictive assumption will be additionally required:
	\begin{enumerate}[label={($\mathcal{O}$)}]\addtolength{\itemindent}{0em}
		\item\label{O} If $N \geq 2$, then $\Omega$ is of class $C^{1,1}$.
	\end{enumerate} 
	
	Throughout the work, we decompose functions into their positive and negative parts as $w = w_+ - w_-$, where $w_\pm := \max\{\pm w, 0\}$, and we denote by $\|\cdot\|_r$ the standard $L^r(\Omega)$-norm, $r \in [1,\infty]$.
	Depending on the context, the weight $m$ will be asked to satisfy one of the following two regularity assumptions:
	\begin{enumerate}
		\addtolength{\itemindent}{0em}
		\myitem[$(\widetilde{\mathcal{M}})$]\label{WM}
		$m_+\not\equiv 0$ and 
		$m \in L^\gamma(\Omega)$ for some $\gamma>N/p$ if $N\ge p$ and 
		$\gamma=1$ if $N<p$. 
		\myitem[$(\mathcal{M})$]\label{M} 
		$m_+\not\equiv 0$ and 
		$m \in L^\gamma(\Omega)$ for some $\gamma>N$. 
	\end{enumerate} 
	These assumptions are motivated by the following facts. 
	Consider the weighted eigenvalue problem for the $p$-Laplacian 
	\begin{equation}\label{eq:EP}
		\left\{
		\begin{aligned}
			-\Delta_p u &= \lambda \,m(x)|u|^{p-2}u
			&&\text{in}\ \Omega, \\
			u&=0 &&\text{on}\ \partial \Omega,
		\end{aligned}
		\right.
	\end{equation}
	and define its first (or principal) positive eigenvalue $\lambda_1(m)$ as
	\begin{equation}\label{eq:lambda1}
		\lambda_1(m)
		:=
		\inf
		\left\{
		\frac{\intO |\nabla u|^p\,dx}{\intO m|u|^p\,dx}:\, u\in\W(\Omega),\ 
		\intO m |u|^p\,dx>0
		\right\}.
	\end{equation}
	Assuming \ref{WM}, it is not hard to see that $\lambda_1(m)$ is attained. 
	Its nonnegative minimizer, which we denote by $\varphi_1$ and which will be naturally referred to as the first eigenfunction, is known to be bounded, locally H\"older continuous, positive in $\Omega$, and unique modulo scaling, see \cite{Cuesta}. 
	Hereinafter, we assume that $\|\varphi_1\|_\infty=1$, for convenience.  
	Under the stronger assumptions \ref{O} and \ref{M}, we have $\varphi_1 \in C^{1,\beta}_0(\overline{\Omega})$ for some $\beta \in (0,1)$ by \cite[Proposition~2.1]{KZ}, see also Proposition~\ref{prop:C1-reg} below for details. 
	For some of our results, it will be important to have the boundary point lemma for $\varphi_{1}$. That is, periodically, we will impose the following assumption in addition to \ref{O} and \ref{M}:
	\begin{enumerate}
		\myitem[$(\varPhi)$]\label{BM}
		$\partial \varphi_1/\partial \nu<0$ on $\partial \Omega$, where 
		$\nu$ is the outer unit normal vector to $\partial \Omega$.
	\end{enumerate} 
	We do not know whether \ref{O} and \ref{M} imply \ref{BM}, and we refer to Remark~\ref{rem:reg-m} for a discussion on sufficient conditions guaranteeing the validity of \ref{BM}. 
	Moreover, we refer to Section~\ref{sec:eigen} for some other properties of the weighted eigenvalue problem~\eqref{eq:EP} needed for the present work.
	
	As for the weight $a$, we will impose one of the following two regularity assumptions:
	\begin{enumerate}
		\addtolength{\itemindent}{0em}
		\myitem[$(\widetilde{\mathcal{A}})$]\label{WA} 
		$a \in L^\gamma(\Omega)\setminus \{0\}$ for some $\gamma>N/p$ if $N\ge p$ and 
		$\gamma=1$ if $N<p$. 
		\myitem[$(\mathcal{A})$]\label{A} 
		$a \in L^\gamma(\Omega) \setminus \{0\}$ for some $\gamma>N$.
	\end{enumerate} 
	The assumption that $a$ is nontrivial is presented in \ref{WA} and \ref{A} without loss of generality, since the case of the identically zero $a$ is covered by taking $\eta=0$.
	If no global restriction on the sign of $m$ or $a$ is imposed, the weight is usually called indefinite.
	
	Finally, the source function $f$ will be required to satisfy either one or few of the following assumptions concerning its regularity and qualitative properties:
	\begin{enumerate}
		\addtolength{\itemindent}{0em}
		\myitem[$(\widetilde{\mathcal{F}})$]\label{WF1} 
		$f \in L^\gamma(\Omega)\setminus \{0\}$ for some $\gamma>N/p$ if $N\ge p$ and 
		$\gamma=1$ if $N<p$. 
		\myitem[$(\mathcal{F})$]\label{F1} 
		$f \in L^\gamma(\Omega) \setminus \{0\}$ for some $\gamma>N$.
		\myitem[$(\mathcal{F}_{\lambda_1})$]\label{F2}
		$\intO f \varphi_1 \,dx > 0$ and the boundary value problem 
		\begin{equation}\label{eq:EPinhom}
			\left\{
			\begin{aligned}
				-\Delta_p u &= \lambda_1(m) \,m(x)|u|^{p-2}u + f(x)
				&&\text{in}\ \Omega, \\
				u&=0 &&\text{on}\ \partial \Omega,
			\end{aligned}
			\right.
		\end{equation}
		does not possess solutions.
	\end{enumerate} 
	Let us observe that if $f$ is nonnegative, then \ref{F2} holds, see, e.g., \cite[Theorem~2.4]{Alleg}, \cite[Corollaire]{FGTT},  \cite[Proposition~4.3]{GGP}, and Corollary~\ref{cor:nonex} below.
	Moreover, in the case $p=2$, the assumption $\intO f \varphi_1 \,dx > 0$ alone guarantees \ref{F2}, as it follows from the Fredholm alternative. 
	In contrast, in the nonlinear case $p \neq 2$, there are examples of $f$ for which $\intO f \varphi_1 \,dx > 0$ and \eqref{eq:EPinhom} \textit{has} a solution, see, e.g., \cite{drabek,takac-lec2,takac-lec1} for an overview.

	The regularity assumptions \ref{WM}, \ref{WA}, \ref{WF1} will be imposed to guarantee that any solution of \eqref{eq:Psub} is bounded and continuous in $\Omega$, see Propositions~\ref{prop:bdd} and~\ref{prop:C0-reg}.
	The stronger regularity assumptions \ref{M}, \ref{A}, \ref{F1}, together with \ref{O}, further guarantee that any solution of \eqref{eq:Psub} belongs to $C^{1,\beta}(\overline{\Omega})$, see \cite{APO,KZ} and Proposition~\ref{prop:C1-reg} below. 
	Clearly, the existence of solutions of \eqref{eq:Psub} can be established under less restrictive assumptions, see Remark~\ref{rem:existence}.

	\subsection{Unperturbed case \texorpdfstring{$\eta = 0$}{eta=0}. AMP}\label{sec:eta=0}
	In the unperturbed case $\eta = 0$, the problem \eqref{eq:Psub} reads as
	\begin{equation}\label{eq:Pfred}
		\left\{
		\begin{aligned}
			-\Delta_p u &= \lambda \,m(x)|u|^{p-2}u + f(x) 
			&&\text{in}\ \Omega, \\
			u&=0 &&\text{on}\ \partial \Omega,
		\end{aligned}
		\right.
	\end{equation}
	and the existence and qualitative properties of solutions of this problem have been subjects of intensive study.
	When $m = 1$ a.e.\ in $\Omega$, 
	the existence theory is fully covered by the classical Fredholm alternative in the linear case $p=2$,
	and we refer to the surveys \cite{drabek,takac-lec2,takac-lec1} and extensive bibliographies therein for a number of involved results on the generalized Fredholm alternative in the nonlinear case $p \neq 2$. 
	Let us explicitly mention that, in contrast to the linear settings, the problem \eqref{eq:Pfred} with $p \neq 2$ might possess several distinct solutions even for nonresonant values of $\lambda$. 
	We also refer to the classical monograph \cite{FNSS} which discusses the Fredholm alternative for general nonlinear operators.
	
	As for the qualitative properties of solutions of the problem \eqref{eq:Pfred}, the value of the parameter $\lambda$ and the sign of the source function $f$ play a crucial role.
	In the case of nonnegative $f$ (and under certain assumptions on $m$), it is well known that any solution of \eqref{eq:Pfred} is \textit{positive} in $\Omega$ for every $\lambda < \lambda_1(m)$, see, e.g., \cite{hess-kato,takac-lec2}. 
	This scenario is called \textit{maximum principle} (MP, for brevity).
	On the other hand, in the case $m = 1$ a.e.\ in $\Omega$, 
	it was first proved in \cite{CP} for $p=2$, and in \cite{FGTT} for $p>1$
	that there exists $\delta>0$ such that any solution of \eqref{eq:Pfred} is \textit{negative} in $\Omega$ when $\lambda \in (\lambda_1(m),\lambda_1(m)+\delta)$.
	This scenario is called \textit{antimaximum principle} (AMP, for brevity).
	The case of indefinite weight $m$ was covered in \cite{AG,IR,GGP,GGP2,hess,pinch}.
	It is also known that the AMP for \eqref{eq:Pfred} is not uniform with respect to $f$, i.e., the maximal value of $\delta$ depends on $f$, see, e.g., \cite[Section~4]{ACG}.
	Estimates on the maximal interval of validity of the AMP have been studied in \cite{BDI,FHT2014}.
	Notice that the MP (locally with respect to $\lambda$) and the AMP are preserved assuming the weaker assumption \ref{F2} instead of $f \geq 0$ a.e.\ in $\Omega$, see \cite[Theorem~17]{AG} for $p=2$ and \cite[Theorem~27]{AG} for $p >1$. 
	We also refer to \cite{M} for a survey.
	
	Let us emphasize that, to the best of our knowledge, in all references on the AMP for the general (nonlinear) $p$-Laplacian, the regularity assumption $m, f \in L^\infty(\Omega)$ is imposed. 
	In view of the regularity result \cite{Lieberman} (and further imposing \ref{O}), this assumption guarantees that any solution of \eqref{eq:Pfred} belongs to the H\"older space $C^{1,\beta}(\overline{\Omega})$ for some $\beta \in (0,1)$ and its $C^{1,\beta}(\overline{\Omega})$-norm is bounded by the $L^\infty(\Omega)$-norm of the right-hand side of \eqref{eq:Pfred}, which is a crucial ingredient for the arguments.
	In contrast, in the linear case $p=2$, it is sufficient to assume that $m, f \in L^\gamma(\Omega)$ with $\gamma>N$, i.e., \ref{M} and \ref{F1},  since by the classical existence and regularity theory (see, e.g., \cite[Theorem~9.15]{GT} for the existence) solutions of \eqref{eq:Pfred} belong to $W^{2,\gamma}(\Omega)$ which is embedded in $C^{1,\beta}(\overline{\Omega})$ continuously for such $\gamma$. 
	(We also refer to Remark~\ref{rem:reg-m} below for a discussion on the necessity of the assumption~\ref{BM}.)
	If $f \in L^N(\Omega)$ and its support is compactly contained in $\Omega$, then the AMP remains valid, see \cite[Therem~1.3]{birin} for the case $m = 1$ a.e.\ in $\Omega$ and compare with Corollary~\ref{cor:AMP-w} below.
	However, in general, the assumption $\gamma > N$ is optimal in the scale of Lebesgue spaces since the AMP is violated for some $f \in L^N(\Omega)$, as shown in \cite{sweers}. 
	An interesting problem in this direction is to investigate optimal regularity assumptions on $f$ in the scale of finer (e.g., Lorentz) spaces, but we do not pursue this question here.
	Instead, one of the main aims of the present work is to justify the fact that the regularity assumptions \ref{M} and \ref{F1} are sufficient for the validity of the AMP for \eqref{eq:Pfred} \textit{in the general nonlinear case} $p>1$. See Corollary~\ref{cor:AMP} below.
	
	Finally, we remark that, in contrast to the maximum principle, the AMP is sensitive to the regularity of the boundary, which is intrinsically connected with the applicability of the boundary point lemma to the first eigenfunction $\varphi_1$.
	A counterexample to the AMP for nonsmooth domains was delivered in \cite[Proposition~3.2]{birin} already in the simplest case $\Omega = (0,\pi)^2$, $p=2$, and $m,f = 1$ a.e.\ in $\Omega$, see also \cite[Section~6]{BS} for a development.
	At the same time, a negativity of solutions continues to persist on compact subsets of $\Omega$, see Corollary~\ref{cor:AMP-loc} below.

\subsection{Zero-source case \texorpdfstring{$f = 0$}{f=0}
}\label{sec:hom}

Although main results of the present work do not apply to the problem \eqref{eq:Psub} with a trivial source function, we briefly review several facts about this case. 
More precisely, when $f = 0$ a.e.\ in $\Omega$, the problem \eqref{eq:Psub} takes the form
\begin{equation}\label{eq:Pfzero}
	\left\{
	\begin{aligned}
		-\Delta_p u &= \lambda \,m(x)|u|^{p-2}u+ \eta \,a(x)|u|^{q-2}u
		&&\text{in}\ \Omega, \\
		u&=0 &&\text{on}\ \partial \Omega.
	\end{aligned}
	\right.
\end{equation}
Without global restrictions on the sign of the weight $a$, the problem \eqref{eq:Pfzero} is called indefinite.
In view of the subhomogeneous nature of the perturbation ($q<p$), sign-constant solutions of \eqref{eq:Pfzero} do not obligatory satisfy the strong maximum principle. 
In particular, there might occur sign-constant (as well as sign-changing) solutions of \eqref{eq:Pfzero} which are zero on open subsets of $\Omega$ called dead cores. 
We refer to \cite{diaz,KQU1} for a detailed discussion and further references on this subject.
In general, the solution set of the problem \eqref{eq:Pfzero} is rich. 
Since the equation in \eqref{eq:Pfzero} is odd with respect to $u$, it is natural to anticipate the existence of infinitely many solutions by minimax arguments, see, e.g., \cite{kaji1}. 
Although information on the sign of solutions obtained by such abstract methods is usually limited, it is expected that most of high-energy solutions (i.e., bound states) of \eqref{eq:Pfzero} are sign-changing. 
Qualitative properties of sign-constant solutions, such as the strict sign vs.\ dead core formation, continuity with respect to parameters, weights, and exponents, uniqueness issues, etc., are of considerable interest and
have been studied, for instance, in \cite{BPT1,DHI1,KQU2,KQU}.
The existence and multiplicity of solutions of \eqref{eq:Pfzero} with respect to parameters $\lambda$ and $\eta$ have been investigated, e.g., in \cite{BobkovTanaka2021,brown,KQU,moroz,QS}.

Our main results on the problem \eqref{eq:Psub} indicate that the presence of the nontrivial source function $f$ significantly changes properties of the solution set of \eqref{eq:Psub} in comparison with that of \eqref{eq:Pfzero}. In particular, in contrast to \eqref{eq:Pfzero}, information on the sign can be deduced for any member of the solution set of \eqref{eq:Psub} in appropriate ranges of the parameters $\lambda$ and $\eta$.

	\section{Main results}\label{sec:main}
	
	The main aim of the present work is the investigation of sign properties of solutions of the problem \eqref{eq:Psub}, in particular, the validity of the maximum principle (MP) and the antimaximum principle (AMP).
	We collect our main results in this direction in Sections~\ref{subsec:mpamp}, \ref{subsec:nonuniform}, \ref{subsec:additional}, and also in Section~\ref{sec:mpampsubsets}.
	Most of the obtained results are valid for the unperturbed problem \eqref{eq:Pfred} and provide new information in this case, see, e.g., Corollaries~\ref{cor:AMP}, \ref{cor:AMP-w}, and~\ref{cor:AMP-loc}.
	
	In addition, in order to justify the existence of solutions of the problem~\eqref{eq:Psub} whose qualitative properties we study, we develop the corresponding existence theory. 
	It is seen from the discussion in Sections~\ref{sec:eta=0} and~\ref{sec:hom} that the structure of the solution set of the problems~\eqref{eq:Pfred} and~\eqref{eq:Pfzero} might be complicated. 
	Naturally, the solution set of \eqref{eq:Psub} may have even more intricate structure and it would be hard to describe it in detail. 
	Therefore, we restrict ourselves only to a general existence result sufficient for our main purposes. 
	The corresponding theorem is given in Section~\ref{subsec:existence}.

	\subsection{MP and AMP}\label{subsec:mpamp}
	
	We start with the following general results on the MP and AMP for the problem \eqref{eq:Psub} which are local with respect to $\lambda$ and $\eta$, see Figure~\ref{fig:1}.
	
	\begin{theorem}\label{thm0}
		\marginnote{{\scriptsize $\qed$ \pageref{page:thm0:proof}}}
		Let \ref{O}, \ref{M}, \ref{BM}, \ref{A}, \ref{F1}, \ref{F2} be satisfied. 
		Assume that $\intO a \varphi_1^q \,dx > 0$.
		Then there exists $\delta>0$ such that any solution $u$ of \eqref{eq:Psub} satisfies $u>0$ in $\Omega$ and $\partial u/\partial \nu < 0$ on $\partial \Omega$ provided $\lambda \in (\lambda_1(m)-\delta,\lambda_1(m))$ and $\eta \in (-\delta,0]$.
	\end{theorem} 
	
	\begin{theorem}\label{thm1}
	\marginnote{{\scriptsize $\qed$ \pageref{page:thm1:proof}}}
		Let \ref{O}, \ref{M}, \ref{BM}, \ref{A}, \ref{F1}, \ref{F2} be satisfied.
		Let one of the following assumptions hold:
		\begin{enumerate}[label={\rm(\Roman*)}]
			\item\label{thm1:I}
			$\intO a \varphi_1^q \,dx > 0$.
			\item\label{thm1:II}
			$\intO a \varphi_1^q \,dx = 0$ and, in addition to $1<q<p$, 
			\begin{equation}\label{eq:Picone-0} 
				(q-1) s^p + q s^{p-1} - (p-q) s + (q-p+1) \geq 0  
				~~\text{for all}~~  s \geq 0.
			\end{equation}
		\end{enumerate}
		Then there exists $\delta>0$ such that any solution $u$ of \eqref{eq:Psub} satisfies $u<0$ in $\Omega$ and $\partial u/\partial \nu > 0$ on $\partial \Omega$ provided $\lambda \in (\lambda_1(m),\lambda_1(m)+\delta)$ and $\eta \in [0, \delta)$.
	\end{theorem} 
	
	\begin{remark}
		We do not know whether Theorem~\ref{thm0} remains valid under the assumption $\intO a \varphi_1^q \,dx = 0$.
		Observe that the cases $\intO a \varphi_1^q \,dx = 0$ and $\intO a \varphi_1^q \,dx > 0$ are of principal importance for Theorems~\ref{thm0}, \ref{thm1} and the results presented below, while the case $\intO a \varphi_1^q \,dx <0$ is reduced to the latter one by considering $(-\eta) \intO (-a) \varphi_1^q \,dx$.   
	\end{remark} 
	
	\begin{remark}\label{rem:m-negative}
		In the case of a nontrivial negative part $m_-$ of $m$, 
		Theorems~\ref{thm0} and~\ref{thm1}, as well as most of the results formulated below, have counterparts for negative values of $\lambda$ when the eigenvalue $\lambda_1(m)$ is replaced by the eigenvalue $-\lambda_1(-m)$ and the first eigenfunction $\varphi_{1}$ is replaced by the first (positive) eigenfunction $\psi_1$ corresponding to $\lambda_1(-m)$.
		In particular, assuming \ref{BM} for $\psi_1$, 
		Theorem~\ref{thm1} is valid for any $\lambda \in (-\lambda_1(-m)-\delta,-\lambda_1(-m))$ with some $\delta>0$, if either the assumption \ref{thm1:I} or \ref{thm1:II} with $\psi_{1}$ instead of $\varphi_{1}$ holds.
	\end{remark}
	
	The case $\eta=0$ in Theorems~\ref{thm0} and~\ref{thm1} corresponds to the MP and AMP for the problem~\eqref{eq:Pfred}, respectively.
	For convenience, we formulate it explicitly.
	\begin{corollary}\label{cor:AMP}
		Let \ref{O}, \ref{M}, \ref{BM}, \ref{F1}, \ref{F2} be satisfied. 
		Then there exists $\delta>0$ such that 
		the following assertions hold:
		\begin{enumerate}[label={\rm(\roman*)}]
			\item\label{cor:AMP:1}
			Any solution $u$ of \eqref{eq:Pfred} satisfies $u>0$ in $\Omega$ and $\partial u/\partial \nu < 0$ on $\partial \Omega$ provided  $\lambda \in (\lambda_1(m)-\delta,\lambda_1(m))$.
			\item\label{cor:AMP:2}
			Any solution $u$ of \eqref{eq:Pfred} satisfies $u<0$ in $\Omega$ and $\partial u/\partial \nu > 0$ on $\partial \Omega$ provided $\lambda \in (\lambda_1(m),\lambda_1(m)+\delta)$.
		\end{enumerate}
	\end{corollary}
	Corollary~\ref{cor:AMP} generalizes the results of \cite[Theorem 27]{AG}, \cite[Th\'eor\`eme~2]{FGTT}, and \cite[Theorem~5.1]{GGP} on the MP and AMP for the problem~\eqref{eq:Pfred} by weakening regularity assumptions on $m$ and $f$. 
		
	\begin{remark}\label{rem:reg-m}
		The statement of \cite[Theorem~17]{AG} on the MP and AMP in the linear case $p=2$ does not explicitly contain the assumption \ref{BM}, but the necessity of \ref{BM} is discussed in \cite[Remark~18]{AG}. 
		We do not know whether \ref{O} and \ref{M} imply \ref{BM} (for both $p=2$ and $p \neq 2$), although we believe that the answer is affirmative. 
		One simple sufficient condition for the validity of \ref{BM} is the following: there exists $\rho>0$ such that $m_-\in L^\infty(\Omega_\rho)$, where 	$\Omega_\rho:=\{x\in\Omega:\,\mathrm{dist}(x,\partial\Omega)<\rho\}$, as it follows from \cite[Theorem~A]{MMT}.		
	\end{remark}

	Based on Corollary~\ref{cor:AMP}, 
	continuity arguments allow to extend the ranges of $\eta$ in Theorems~\ref{thm0} and \ref{thm1} to some $\eta>0$ and $\eta<0$, respectively, even without any sign assumptions on $\intO a \varphi_{1}^q \, dx$.
	\begin{theorem}\label{thm-1}
	\marginnote{{\scriptsize $\qed$ \pageref{page:thm-1:proof}}}
		Let \ref{O}, \ref{M}, \ref{BM}, \ref{A}, \ref{F1}, \ref{F2} be satisfied.
		Then there exists $\delta>0$ such that the following assertions hold:
		\begin{enumerate}[label={\rm(\roman*)}]
			\item\label{thm-1:1}
			For any $\lambda \in (\lambda_1(m)-\delta,\lambda_1(m))$ there exists $\overline{\eta}_\lambda>0$ such that any solution $u$ of \eqref{eq:Psub} satisfies $u>0$ in $\Omega$ and $\partial u/\partial \nu < 0$ on $\partial \Omega$ provided $\eta \in (-\overline{\eta}_\lambda, \overline{\eta}_\lambda)$.
			\item\label{thm-1:2}
			For any $\lambda \in (\lambda_1(m),\lambda_1(m)+\delta)$ there exists $\underline{\eta}_\lambda<0$ such that any solution $u$ of \eqref{eq:Psub} satisfies $u<0$ in $\Omega$ and $\partial u/\partial \nu > 0$ on $\partial \Omega$ provided $\eta \in (\underline{\eta}_\lambda, -\underline{\eta}_\lambda)$.
		\end{enumerate}
	\end{theorem}
	
	\begin{remark}\label{rem:sufficient}
		If, in addition to \ref{M}, \ref{A}, \ref{F1}, we assume that $m,a,f \in L^\infty(\Omega)$, then Theorems~\ref{thm0}, \ref{thm1}, \ref{thm-1}, and Corollary~\ref{cor:AMP} remain valid under a slightly weaker assumption on $\Omega$ than \ref{O}. 
		Namely, it is sufficient to assume that, in the case $N \geq 2$, $\Omega$ is of class $C^{1,\alpha}$ for some $\alpha \in (0,1)$. 
		Moreover, under these requirements, it is not necessary to impose \ref{BM} in advance.		
		Indeed, the boundedness of $m,a,f$ and the $C^{1,\alpha}$-regularity of $\Omega$ guarantee that any solution of \eqref{eq:Psub}, as well as $\varphi_{1}$, belong to $C^{1,\beta}(\overline{\Omega})$ with the same estimate for the $C^{1,\beta}(\overline{\Omega})$-norm as in the key Proposition~\ref{prop:C1-reg}, see \cite{Lieberman}, and $\varphi_1$ satisfies the boundary point lemma, i.e., the assumption \ref{BM}, see \cite{melkshah}. 
		These facts are main ingredients for the proofs of Theorems~\ref{thm0}, \ref{thm1}, and \ref{thm-1}.
	\end{remark}
	
	The MP and AMP without information on the behavior of solutions on (or near) the boundary $\partial\Omega$ can be obtained under weaker regularity assumptions on the parameters of \eqref{eq:Psub} and additional assumptions on the behavior of $a$ and $f$ near the boundary of $\Omega$. 
	Recall the notation
	$$
	\Omega_\rho 
	= 
	\{x \in \Omega:\, \mathrm{dist}(x,\partial \Omega) < \rho\}.
	$$
	\begin{theorem}\label{thm1-w}
		Let
		\ref{WM}, \ref{WA}, \ref{WF1}, \ref{F2} be satisfied. 
		Assume that $\intO a \varphi_1^q \,dx > 0$ and there exists $\rho>0$ such that $a = 0$ a.e.\ in $\Omega_\rho$.
		Then the following 
		assertions hold:
		\begin{enumerate}[label={\rm(\roman*)}]
			\item\label{thm1-w:1}
			\marginnote{{\scriptsize $\qed$ \pageref{page:thm-1w:1:proof}}}
			Assume that $f \geq 0$ a.e.\ in $\Omega_\rho$.
			Then there exists $\delta>0$ such that 
			any solution $u$ of \eqref{eq:Psub} satisfies $u>0$ in $\Omega$ provided
			$\lambda \in (\lambda_1(m)-\delta,\lambda_1(m))$  and $\eta \in (-\delta,0]$.
			\item\label{thm1-w:2}
			\marginnote{{\scriptsize $\qed$ \pageref{page:thm-1w:2:proof}}}
			Assume that $f = 0$ a.e.\ in $\Omega_\rho$. 
			Then there exists $\delta>0$ such that any solution $u$ of \eqref{eq:Psub} satisfies $u<0$ in $\Omega$ provided 
			$\lambda \in (\lambda_1(m),\lambda_1(m)+\delta)$  and $\eta \in [0,\delta)$.
		\end{enumerate}
	\end{theorem} 

	In the unperturbed linear case (i.e., $\eta=0$ and $p=2$), the AMP under the assumptions $f \in L^N(\Omega)$ 
	and $\mathrm{supp}\, f \subset \Omega$ was obtained in \cite[Theorem~1.3]{birin}.
	The regularity assumption on $f$ in Theorem~\ref{thm1-w}  is weaker, which gives, therefore,  new information on the AMP already  for $p=2$.
	We believe that our arguments for Theorem~\ref{thm1-w} can be generalized to cover even less regular weights and the source function.
	Let us also observe that the general nonnegativity of $f$ in $\Omega_\rho$ is not enough to guarantee that the AMP in  Theorem~\ref{thm1-w}~\ref{thm1-w:2} is satisfied, see a counterexample given by  \cite[Proposition~3.2]{birin}. 
	However, an appropriate control of the growth or decay of $f$ near irregular parts of $\partial\Omega$ might result in the validity of the AMP, see \cite[Theorem~11]{BS}.
	
	For reader's convenience, we provide the explicit formulation of Theorem~\ref{thm1-w} for the unperturbed case $\eta=0$, i.e., for the problem \eqref{eq:Pfred}.
	\begin{corollary}\label{cor:AMP-w}
		Let
		\ref{WM}, \ref{WF1}, \ref{F2} be satisfied and $\rho>0$.
		Then the following 
		assertions hold:
		\begin{enumerate}[label={\rm(\roman*)}]
			\item\label{cor:AMP-w:1}
			Assume that $f \geq 0$ a.e.\ in $\Omega_\rho$.
			Then there exists $\delta>0$ such that any solution $u$ of \eqref{eq:Pfred} satisfies $u>0$ in $\Omega$ provided 			 
			$\lambda \in (\lambda_1(m)-\delta,\lambda_1(m))$.
			\item\label{cor:AMP-w:2}
			Assume that $f = 0$ a.e.\ in $\Omega_\rho$. 
			Then there exists $\delta>0$ such that any solution $u$ of \eqref{eq:Pfred} satisfies $u<0$ in $\Omega$ provided $\lambda \in (\lambda_1(m),\lambda_1(m)+\delta)$.
		\end{enumerate}
	\end{corollary}

	Finally, in analogy with Theorem~\ref{thm-1}, we provide the following result on a certain extension of the ranges of $\eta$ in Theorem~\ref{thm1-w}, which does not require sign assumptions on $\intO a \varphi_{1}^q \, dx$.
	\begin{theorem}\label{thm-1ww}
	\marginnote{{\scriptsize $\qed$ \pageref{page:thm-1ww:proof}}}
		Let
		\ref{WM}, \ref{WA}, \ref{WF1}, \ref{F2} be satisfied. 
		Assume that there exists $\rho>0$ such that $a = 0$ a.e.\ in $\Omega_\rho$.
		Then the following assertions hold:
		\begin{enumerate}[label={\rm(\roman*)}]
			\item\label{thm-1ww:1}
			Assume that $f \geq 0$ a.e.\ in $\Omega_\rho$. 
			Then there exists $\delta>0$ such that for any $\lambda \in (\lambda_1(m)-\delta,\lambda_1(m))$ there exists $\overline{\eta}_\lambda>0$ such that any solution $u$ of \eqref{eq:Psub} satisfies $u>0$ in $\Omega$ provided $\eta \in (-\overline{\eta}_\lambda, \overline{\eta}_\lambda)$.
			\item\label{thm-1ww:2}
			Assume that $f = 0$ a.e.\ in $\Omega_\rho$. 
			Then there exists $\delta>0$ such that for any $\lambda \in (\lambda_1(m),\lambda_1(m)+\delta)$ there exists $\underline{\eta}_\lambda<0$ such that any solution $u$ of \eqref{eq:Psub} satisfies $u<0$ in $\Omega$ provided $\eta \in (\underline{\eta}_\lambda, -\underline{\eta}_\lambda)$.
		\end{enumerate}
	\end{theorem}

	\subsection{Nonuniformity of AMP}\label{subsec:nonuniform}
	In the unperturbed case $\eta=0$, it is well known that the AMP is not uniform with respect to $f$. 
	That is, the maximal value of $\delta>0$ defining the interval $(\lambda_1(m),\lambda_1(m)+\delta)$ of validity of the AMP for the problem \eqref{eq:Pfred} depends on $f$ and can be made as small as desired. 
	We refer to \cite[Theorems~5.1 (ii) and 5.2 (i)]{GGP}
	for explicit statements. 
	In the following theorem, we generalize this fact to the case of the problem \eqref{eq:Psub} and improve it by weakening regularity assumptions.
	\begin{theorem}\label{thm3} 
	\marginnote{{\scriptsize $\qed$ \pageref{page:thm3:proof}}}
		Let \ref{WM}, \ref{WA} be satisfied. Assume that $a\ge 0$ a.e.\ in $\Omega$.
		Then for any $\varepsilon>0$ there exists $f \in C_0^\infty(\Omega)\setminus \{0\}$ satisfying $f \geq 0$ in $\Omega$ such that 
		\eqref{eq:Psub} has no nonnegative solution  and no negative solution
		provided $\lambda\ge \lambda_1(m)+\varepsilon$ and $\eta\ge 0$. 
	\end{theorem} 
Theorem~\ref{thm3} implies that for any solution $u$ of \eqref{eq:Psub} (with corresponding parameters) there exists $x_0 \in \Omega$ such that $u(x_0) = 0$. That is, $u$ is either nonpositive (but not negative) or sign-changing.  
We also refer to Proposition~\ref{prop:non-exists-0} below for a more general result on the nonexistence of nonnegative solutions.

\begin{figure}[h!]
	\center{
		\includegraphics[width=0.7\linewidth]{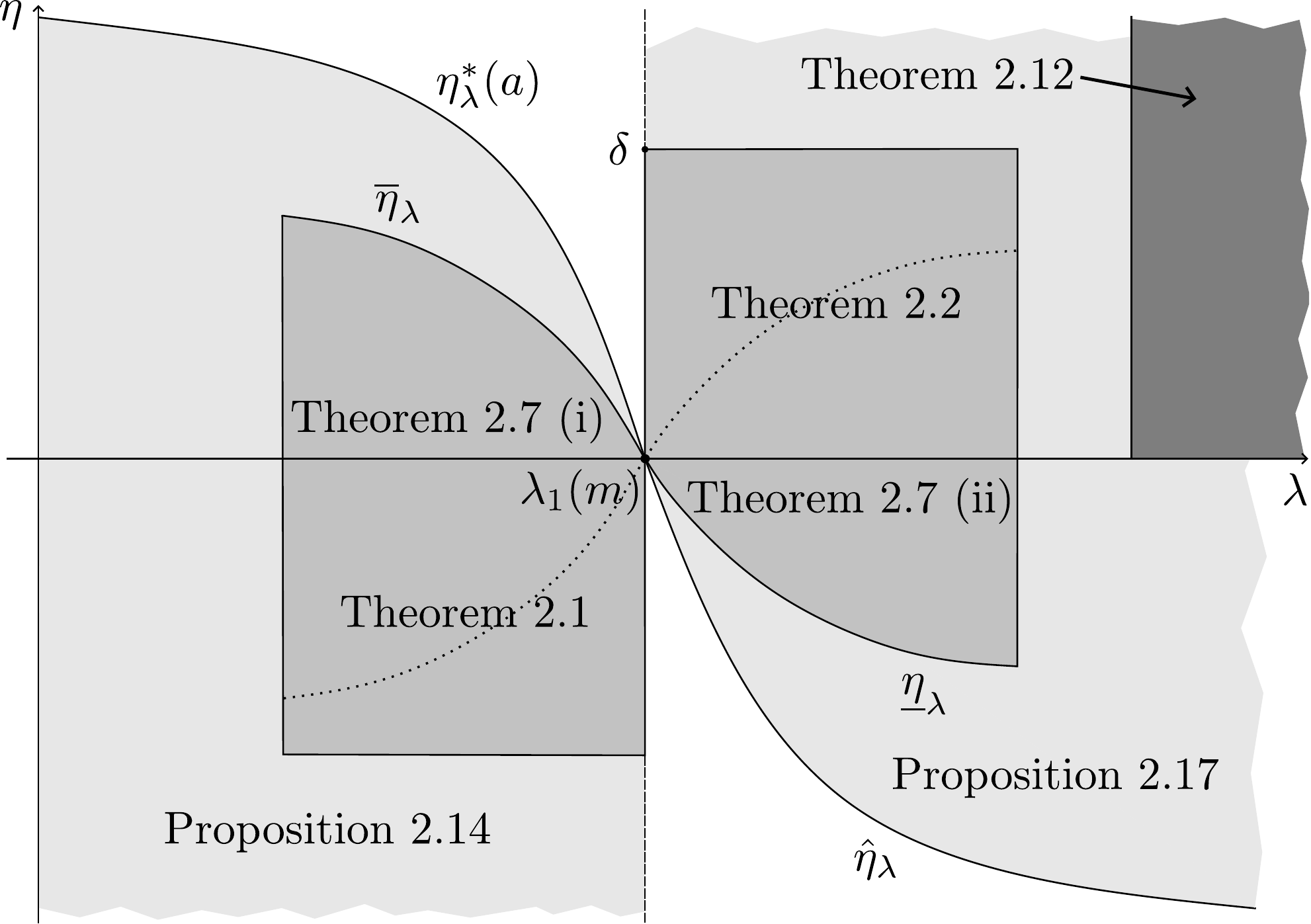}
		\caption{A schematic plot of the main results under the assumptions \ref{O}, \ref{M}, \ref{BM}, \ref{A}, \ref{F1} and $a,f \geq 0$ a.e.\ in $\Omega$.}
		\label{fig:1}
	}
\end{figure}

The nonnegativity of the weight $a$ required in Theorem~\ref{thm3} can be weakened by imposing additional assumptions on other parameters of \eqref{eq:Psub}.
	
	\begin{theorem}\label{thm5} 
	\marginnote{{\scriptsize $\qed$ \pageref{page:thm5:proof}}}
	Let \ref{O}, \ref{M}, \ref{A} be satisfied.
	Assume that  
	$m \ge 0$ a.e.\ in $\Omega$ 
	and $\intO a\varphi_1^q\,dx>0$. 
	Assume also, in addition to
	$1<q<p$, that
	\begin{equation}\label{eq:Picone-01} 
		(q-1) s^p + q s^{p-1} - (p-q) s + (q-p+1) \geq 0  
		~~\text{for all}~~  s \geq 0.
	\end{equation}
	Then for any $\lambda>\lambda_1(p)$ 
	there exists $f \in C_0^\infty(\Omega)\setminus \{0\}$ satisfying $f \geq 0$ in $\Omega$ 
	such that 
	\eqref{eq:Psub} has no positive solution and no negative solution 
	provided $\eta\ge 0$. 
\end{theorem}
	
	It is clear that in the unperturbed case $\eta=0$ the result of Theorem~\ref{thm3} is stronger than that of Theorem~\ref{thm5}.

	\subsection{Additional properties}\label{subsec:additional}
	Let us collect a few additional qualitative properties of solutions of the problem \eqref{eq:Psub}.
	They provide less precise results compared to Theorems~\ref{thm0}, \ref{thm1}, and \ref{thm-1}, but ask for lower regularity assumptions and cover larger regions of $\lambda$ and $\eta$, see Figure~\ref{fig:1}.
	
	\begin{proposition}\label{prop:non-exists-01}\marginnote{{\scriptsize $\qed$ \pageref{page:prop:non-exists-01:proof}}} 
		Let \ref{WM}, \ref{WA}, \ref{WF1} be satisfied. 
		Assume that $f \geq 0$ a.e.\ in $\Omega$. 
		Then any solution of \eqref{eq:Psub} is nonnegative provided $0 \leq \lambda \leq \lambda_1(m)$ and either 
		$-\eta^*_\lambda(-a) < \eta \leq 0$ 
		or 
		$0 \leq \eta < \eta^*_\lambda(a)$, 
		where the critical value $\eta^*_\lambda(a) \geq 0$ is defined as follows:
		\begin{align}
			\eta^*_\lambda(a)
			&=
			\frac{p-1}{(p-q)^\frac{p-q}{p-1}(q-1)^\frac{q-1}{p-1}}\\
			\label{eq:eta*1}
			&\times
			\inf
			\left\{
			\frac{\left(\intO |\nabla u|^p \,dx - \lambda \intO m u^p \,dx\right)^\frac{q-1}{p-1} \left(\intO f u \,dx\right)^\frac{p-q}{p-1}}{\intO a u^q \,dx}:\, u \in \Theta(a)
			\right\},\\
			\label{eq:Theta}
			\Theta(a) 
			&= 
			\left\{
			u \in \W(\Omega):\, u \geq 0, ~ \intO a u^q \,dx > 0
			\right\},
		\end{align}
		and we set $\eta^*_\lambda(a) = +\infty$ if $\Theta(a) = \emptyset$.
	\end{proposition}
	
	\begin{remark}\label{rem:eta}
		It is not hard to see that the functional on the right-hand side of \eqref{eq:eta*1} is $0$-homogeneous with respect to $u$. 
		Clearly, if there exists $u \in \Theta(a)$ such that $\intO f u \,dx = 0$, then $\eta^*_\lambda(a)=0$.
		The same is true if $\lambda=\lambda_1(m)$ and $\intO a \varphi_{1}^q \,dx > 0$, by taking $u=\varphi_1$.
		In Lemma~\ref{lem:eta-est}, we provide sufficient assumptions guaranteeing that $\eta^*_\lambda(a) > 0$ for $0 \leq \lambda<\lambda_1(m)$.
	\end{remark}

	\begin{remark}
		If $m \geq 0$ a.e.\ in $\Omega$, then the weak maximum principle of Proposition~\ref{prop:non-exists-01} remains valid for any $\lambda < 0$. 
		On the other hand, if $m_-$ is nontrivial, then Proposition~\ref{prop:non-exists-01} holds true when $-\lambda(-m) \leq \lambda < 0$, see also Remark~\ref{rem:m-negative}.
		The existence of sign-changing solutions of a particular case of \eqref{eq:Psub} with $p=2$, $\lambda=0$, continuous positive $a$, and continuous $f$ given by \cite[Theorem~1.3]{LLZ} suggests that the assertion of Proposition~\ref{prop:non-exists-01} cannot be extended for sufficiently large $\eta>0$.
		Finally,  Proposition~\ref{prop:non-exists-01} is not generally true if the subhomogeneous assumption $q<p$ is replaced by $q>p$, which is indicated by \cite[Theorem~A]{NC}.
	\end{remark}
	
	\begin{proposition}\label{prop:non-exists-0} 
	\marginnote{{\scriptsize $\qed$ \pageref{page:prop:non-exists-0:proof}}} 
		Let \ref{WM}, \ref{WA}, \ref{WF1} be satisfied.
		Assume that $a, f \geq 0$ a.e.\ in $\Omega$.
		Then for any $\lambda\ge \lambda_1(m)$ there exists $\hat{\eta}_\lambda \leq 0$ such that 	
		\eqref{eq:Psub} has no nonnegative solutions provided $\eta \geq \hat{\eta}_\lambda$. 
		Moreover, if $\lambda > \lambda_1(m)$, then $\hat{\eta}_\lambda < 0$.
	\end{proposition}
	
	\begin{remark}
		In the context of Proposition~\ref{prop:non-exists-0}, it is natural to ask for assumptions on $m$, $a$, $f$ which provide more precise information on the sign of solutions of \eqref{eq:Psub}. 
		In \cite[Theorem~17 (3)]{AG},  concerning the linear case $p=2$, it is stated that the assumptions $m \geq 0$ a.e.\ in $\Omega$ and $\intO f \varphi_{1} \, dx = 0$ imply that any solution of \eqref{eq:Pfred} with  $\lambda \neq \lambda_1(m)$ is sign-changing, which should follow from the equality $(\lambda_1(m)-\lambda) \intO m u \varphi_{1} \,dx =0$, according to the proof.
		However, the nonnegativity of $m$ is not enough to make such a conclusion, since one could imagine a sign-constant solution $u$ whose support is located in the zero set of $m$ (assuming that the latter one has a nonempty interior).	
		If the stronger assumption $m > 0$ a.e.\ in $\Omega$ is imposed, then the result is indeed correct, and it is interesting to know if it is true under the original assumption that $m$ is nonnegative. 
	\end{remark}

	Finally, we discuss a sufficient assumption on $f$ guaranteeing the validity of \ref{F2}.
	The following result is a corollary of Propositions~\ref{prop:non-exists-01} and~\ref{prop:non-exists-0} with $\lambda=\lambda_1(m)$ and $\eta=0$, and it provides an improvement of \cite[Theorem~2.4]{Alleg}, \cite[Corollaire]{FGTT}, and \cite[Proposition~4.3]{GGP}.
	\begin{corollary}\label{cor:nonex}
		Let \ref{WM}, \ref{WF1} be satisfied. 
		Assume that $f \geq 0$ a.e.\ in $\Omega$.
		Then \ref{F2} holds.
		In particular, \eqref{eq:EPinhom} does not possess solutions.
	\end{corollary}

\subsection{Existence of solutions}\label{subsec:existence}
In order to justify that solutions of the problem \eqref{eq:Psub} whose properties we discussed in the previous subsections do exist, we provide one general result in this direction.
Observe that 
the problem \eqref{eq:Psub} is variational in the sense that it has an associated 
energy functional $\E \in C^1(\W(\Omega), \mathbb{R})$ whose critical points are  solutions of \eqref{eq:Psub}.
The functional $\E$ is defined as
\begin{equation*}
	\E(u) = \frac{1}{p}\, H_\lambda (u)
	-\frac{\eta}{q}\intO a|u|^q\,dx 
	- \intO f u \,dx, 
	\quad u \in \W(\Omega),
\end{equation*}
where 
$$
H_\lambda(u)
:=
\intO |\nabla u|^p\,dx-\lambda\intO m|u|^p\,dx.
$$
More precisely, under a solution of \eqref{eq:Psub} we mean a function $u \in \W(\Omega)$ which satisfies
\begin{equation*}
	\langle \E'(u),\xi\rangle 
	=
	\intO |\nabla u|^{p-2} \nabla u \nabla \xi \,dx
	-
	\lambda\intO m|u|^{p-2} u \xi\,dx
	-\eta\intO a|u|^{q-2}u \xi\,dx 
	- \intO f \xi \,dx = 0
\end{equation*}
for all $\xi \in \W(\Omega)$.

Let $\sigma(-\Delta_p\,;m)$ stand for the set of all eigenvalues of the problem \eqref{eq:EP}, i.e., its spectrum.
Let us denote the eigenspace corresponding to $\lambda \in \sigma(-\Delta_p\,;m)$ as $ES(\lambda;m)$, that is, 
\begin{equation}\label{eq:ES}
	ES(\lambda;m)
	:=
	\{
	u\in \W(\Omega):\, u\ {\rm is\ a\ solution\ of}\ 
	\eqref{eq:EP}
	\}.
\end{equation}
Now we are ready to formulate the existence result.
\begin{theorem}\label{thm:existence} 
	\marginnote{{\scriptsize $\qed$ \pageref{page:thm:existence:proof}}} 
	Let \ref{WM}, \ref{WA}, \ref{WF1} be satisfied. 	
	Let either of the following assumptions hold: 
	\begin{enumerate}[label={{\rm {(\roman*)}}}]	
		\addtolength{\itemindent}{0em}
		\item\label{thm:existence:1} $\lambda\not\in\sigma(-\Delta_p\,;m)$; 
		\item\label{thm:existence:2}
		$\lambda\in\sigma(-\Delta_p\,;m)$ and 
		$\eta\intO a|u|^q\,dx>0$ 
		for all $u\in ES(\lambda;m)\setminus\{0\}$; 
		\item\label{thm:existence:3}
		$\lambda\in\sigma(-\Delta_p\,;m)$ and 
		$\eta\intO a|u|^q\,dx<0$ 
		for all $u\in ES(\lambda;m)\setminus\{0\}$. 
	\end{enumerate}
	Then $\E$ has at least one critical point, i.e., the problem \eqref{eq:Psub} has at least one solution. 
\end{theorem} 

The proof of Theorem \ref{thm:existence} is based on the linking method, and we refer to \cite{BobkovTanaka2019,Tanaka,TM} for related results.

\begin{remark}\label{rem:existence}
	In Theorem~\ref{thm:existence}, as well as in auxiliary Lemmas~\ref{lem:PS0}, \ref{lem:PS}, \ref{lem:convsol} below, the assumptions \ref{WA} and \ref{WF1} are not optimal,  but we keep them for simplicity and uniformity with other results. 
	In fact, Theorem~\ref{thm:existence} and Lemmas~\ref{lem:PS0}, \ref{lem:PS}, \ref{lem:convsol} remain valid if, instead of \ref{WA}, we assume $a \in L^\gamma(\Omega) \setminus\{0\}$ for some $\gamma > Np/(N(p-q)+pq)$ if $N \geq p$ and $\gamma=1$ if $N <p$, and, instead of \ref{WF1}, we assume $f \in (\W(\Omega))^*\setminus \{0\}$. 
\end{remark}

\medskip
The rest of the article has the following structure.
In Section~\ref{sec:qual}, we provide several auxiliary assertions needed to prove our main results.
Section~\ref{sec:mpampsubsets} contains two propositions about ``local'' versions of the MP and AMP on compact subsets of $\Omega$. These results are used to prove Theorem~\ref{thm1-w} but also have an independent interest.
In Section~\ref{sec:proofs}, we give the proofs of all our main results regarding qualitative properties of solutions.
Theorem~\ref{thm:existence} on the existence of solutions is proved in Section~\ref{sec:existence}.
Appendix~\ref{sec:regularity} contains two regularity results which we often employ in the arguments.
Finally, in Appendix~\ref{sec:picone}, we provide a version of the Picone inequality which is convenient to apply in the weak settings.

\section{Few auxiliary results}\label{sec:qual}
In this section, we collect several auxiliary assertions 
needed to prove our main results stated in Section~\ref{sec:main} and also Section~\ref{sec:mpampsubsets}.

\subsection{Convergences}\label{subsec:conv}

We start with compactness-type results.
Recall the notation \eqref{eq:ES} for the eigenspace $ES(\lambda;m)$ corresponding to an eigenvalue $\lambda \in \sigma(-\Delta_p\,;m)$.
We will use the notation
\begin{equation}\label{eq:pstar}
p^* := \frac{Np}{N-p} 
~~\text{if}~~ N>p
\quad \text{and} \quad 
p^* := \infty 
~~\text{if}~~ N \leq p,
\end{equation}
and denote by $\|\cdot\|_*$ the operator norm.
\begin{lemma}\label{lem:PS0}
	Let \ref{WM}, \ref{WA}, \ref{WF1} be satisfied. 
	Let $\{u_n\}$ be a bounded Palais--Smale sequence for $\E$. 
	Then $\{u_n\}$ converges in $\W(\Omega)$ to a critical point of $\E$, up to a subsequence.
\end{lemma}
\begin{proof}
	The proof is standard, but since similar arguments are used in several subsequent lemmas, we provide a few details here in order to skip them later.
	In view of the boundedness in $\W(\Omega)$, the sequence $\{u_n\}$ 
	converges (along a subsequence) to some $u_0 \in \W(\Omega)$ weakly in $\W(\Omega)$ and strongly in $L^r(\Omega)$ whenever $1 \leq r < p^*$ if $N \geq p$ and $1 \leq r \leq p^*$ if $N < p$.
	The convergence $\|\E'(u_n)\|_* \to 0$ implies that
	\begin{align*}
		\intO |\nabla u_n|^{p-2}\nabla u_n\nabla (u_n-u_0)\,dx
		&-
		\lambda \intO m |u_n|^{p-2} u_n (u_n-u_0)\,dx\\
		&-\eta \intO a |u_n|^{q-2} u_n (u_n-u_0)\,dx 
		- \intO f (u_n-u_0) \,dx 
		\to 0 
	\end{align*}
	as $n\to \infty$. 
	Since the assumptions \ref{WM}, \ref{WA}, \ref{WF1} give the inequalities
	$q\gamma/(\gamma-1)<p\gamma/(\gamma-1)<p^*$ when $N \geq p$, we deduce that
	\begin{equation}\label{eq:lem:converg1}
	\intO m |u_n|^{p-2}u_n(u_n-u_0)\,dx \to 0,
	\quad
	\intO a |u_n|^{q-2}u_n(u_n-u_0)\,dx \to 0,
	\quad
	\intO f (u_n-u_0) \, dx \to 0
	\end{equation}
as $n \to \infty$, which yields
	$$
	\lim_{n \to \infty} \intO |\nabla u_n|^{p-2} \nabla u_n \nabla(u_n-u_0) \,dx = 0.
	$$
	Therefore, the $(S_+)$-property of the $p$-Laplacian (see, e.g., \cite[Theorem 10]{dinica}) guarantees that $u_n \to u_0$ strongly in $\W(\Omega)$, and hence we easily conclude that $u_0$ is a critical point of $\E$.
\end{proof}

\begin{lemma}\label{lem:PS} 
	Let \ref{WM}, \ref{WA}, \ref{WF1} be satisfied. 
	Let either of the assumptions \ref{thm:existence:1}, \ref{thm:existence:2}, \ref{thm:existence:3} of Theorem~\ref{thm:existence} holds. 
	Then $\E$ satisfies the Palais--Smale condition. 
\end{lemma}
\begin{proof} 
	Let $\{u_n\} \subset \W(\Omega)$ be a Palais--Smale sequence for $\E$. 
	In view of Lemma~\ref{lem:PS0}, it is sufficient to show that $\{u_n\}$ is bounded. 
	Suppose, by contradiction, that $\|\nabla u_n\|_p \to \infty$  along a subsequence. Then, arguing in the same way as in \protect{\cite[Lemma 2.24]{BobkovTanaka2021}}, 
	we see that the sequence consisted of normalized functions
	$v_n = u_n/\|\nabla u_n\|_p$ converges strongly in $\W(\Omega)$ to an eigenfunction $v_0\in ES(\lambda;m)\setminus \{0\}$, up to a subsequence.  
	Hence, we obtain a contradiction whenever 
	$\lambda\not\in\sigma(-\Delta_p;m)$. 
	Assume now that either the assumption \ref{thm:existence:2} or \ref{thm:existence:3} of Theorem~\ref{thm:existence} holds. 
	Letting $n\to\infty$ in 
	\begin{align*} 
		o(1) 
		&=\frac{1}{\|\nabla u_n\|_p^q}
		\left(p\E(u_n) -\left< \E^\prime (u_n),u_n\right>\right) \\
		&=\left(1-\frac{p}{q}\right)\eta \intO a |v_n|^q\,dx 
		-\frac{p-1}{\|\nabla u_n\|_p^{q-1}}\,
		\intO f v_n \, dx,
	\end{align*}
	we get 
	$\eta \intO a|v_0|^q\,dx=0$, which is impossible. 
	Thus, $\{u_n\}$ is bounded in $\W(\Omega)$, which completes the proof in view of Lemma~\ref{lem:PS0}.
\end{proof}

\begin{lemma}\label{lem:convsol}
	Let \ref{WM}, \ref{WA}, \ref{WF1} be satisfied.
	Let $\{\lambda_n\}, \{\eta_n\} \subset \mathbb{R}$ be arbitrary convergent sequences. 
	Denote $\lambda := \lim_{n \to \infty} \lambda_n$ and $\eta := \lim_{n \to \infty} \eta_n$. 
	Let $u_n \in \W(\Omega)$ be a solution of {\renewcommand{\Plef}{\mathcal{P};\lambda_n,\eta_n}\eqref{eq:Psub}}. 
	If $\{u_n\}$ is bounded in $\W(\Omega)$, then it converges in $\W(\Omega)$ to a solution of {\renewcommand{\Plef}{\mathcal{P};\lambda,\eta}\eqref{eq:Psub}}, up to a subsequence. 
\end{lemma}
\begin{proof}
	Since $\{\lambda_n\}$ and $\{\eta_n\}$ are convergent and $\{u_n\}$ is bounded in $\W(\Omega)$, we see that $\{E_{\lambda,\eta}(u_n)\}$ is bounded.
	Noting that $u_n$ is a critical point of $E_{\lambda_n,\eta_n}$, we have
	\begin{align}
		\|\E^\prime(u_n)\|_*
		&= 
		\|\E^\prime(u_n)-E_{\lambda_n,\eta_n}^\prime(u_n)\|_*
		\\
		&=
		\sup
		\left\{
		\frac{-(\lambda-\lambda_n)\intO m u_n^{p-1} v\,dx - (\eta-\eta_n)\intO a u_n^{q-1} v\,dx}{\|\nabla v\|_p}:\, v \in \W(\Omega) \setminus \{0\}
		\right\}
		\\
		\label{eq:Enorm}
		&\leq 
		C |\lambda-\lambda_n| \|m\|_\gamma \|u_n\|_{p\gamma/(\gamma-1)}^{p-1} 
		+
		C |\eta-\eta_n| \|a\|_\gamma \|u_n\|_{q\gamma/(\gamma-1)}^{q-1},
	\end{align}
	where $C>0$ does not depend on $u_n$.
	Thanks to the assumptions \ref{WM} and \ref{WA}, we see that $\{u_n\}$ is a bounded Palais--Smale sequence for $\E$.
	Hence, Lemma~\ref{lem:PS0} guarantees that $\{u_n\}$ converges in $\W(\Omega)$ to a critical point of $\E$, up to a subsequence.
	This critical point is a solution of \eqref{eq:Psub}.
\end{proof}

\begin{lemma}\label{lem:auxil0}
	Let \ref{WM}, \ref{WA}, \ref{WF1} be satisfied.
	Let $\{\lambda_n\}, \{\eta_n\} \subset \mathbb{R}$ be arbitrary convergent sequences. 
	Let $u_n \in \W(\Omega)$ be a solution of {\renewcommand{\Plef}{\mathcal{P};\lambda_n,\eta_n}\eqref{eq:Psub}}.
	If $\|\nabla u_n\|_p \to \infty$, then  $\|u_n\|_\infty \to \infty$, 
	the sequence $\{\lambda_n\}$ converges to an eigenvalue $\lambda$ of the problem \eqref{eq:EP}, 
	and the normalized sequence $\{u_n/\|u_n\|_\infty\}$ converges in $\W(\Omega)$ and $C^0_{\mathrm{loc}}(\Omega)$ to an eigenfunction associated with the eigenvalue $\lambda$, up to a subsequence. 
\end{lemma}
\begin{proof}
	Taking $u_n$ as a test function for {\renewcommand{\Plef}{\mathcal{P};\lambda_n,\eta_n}\eqref{eq:Psub}}, we obtain 
	\begin{align}
		\notag 
		\|\nabla u_n\|_p^p 
		&=\lambda_n\intO m|u_n|^p\,dx+\eta_n \intO a|u_n|^q\,dx+\intO fu_n\,dx \\
		\label{eq:lem:conv1}
		&\le |\lambda_n| \|m\|_1\,\|u_n\|_\infty^p 
		+|\eta_n| \|a\|_1\,\|u_n\|_\infty^q+\|f\|_1\|u_n\|_\infty.
	\end{align}
	This shows that the divergence $\|\nabla u_n\|_p \to \infty$ implies $\|u_n\|_\infty \to \infty$.	
	Consider a sequence of normalized  functions $v_n = u_n/\|u_n\|_\infty$. 
	The estimate \eqref{eq:lem:conv1} gives the boundedness of $\{v_n\}$ in $\W(\Omega)$.
	In particular, there exists $v_0 \in \W(\Omega)$ such that $v_n \to v_0$ weakly in $\W(\Omega)$. 
	Similarly to the proof of Lemma~\ref{lem:PS0}, the $(S_+)$-property of the $p$-Laplacian guarantees that $v_n \to v_0$ strongly in $\W(\Omega)$.
	If $N \geq p$, we apply Proposition~\ref{prop:bdd} to the solutions $u_n$ and, dividing the inequality \eqref{eq:linfty} by $\|u_n\|_\infty$, we get
	$$
	1 
	\leq 
	C
	\left(
	\frac{1}{\|u_n\|_\infty} + \|v_n\|_r
	\right)
	$$
	for an appropriate $r$,
	which implies that $v_0$ is nontrivial.
	The same is true if $N<p$ by applying the Morrey lemma.
	
	Let us prove the convergence $v_n \to v_0$ in $C^0_{\mathrm{loc}}(\Omega)$. 
	Denote
	\begin{equation}\label{eq:tildegnx}
		\tilde{g}_n(x) = \lambda_n \,m(x)|v_n(x)|^{p-2}v_n(x)
		+ 
		\frac{\eta_n \,a(x)|v_n(x)|^{q-2}v_n(x)}{\|u_n\|_\infty^{p-q}} + \frac{f(x)}{\|u_n\|_\infty^{p-1}},
		\quad x \in \Omega.
	\end{equation}
	That is, each $v_n$ weakly solves the problem
	\begin{equation}\label{eq:lem:v}
		-\Delta_p v_n = \tilde{g}_n(x) 
		\quad \text{in}\ \Omega, \quad 
		v_n=0 \quad \text{on}\ \partial \Omega.
	\end{equation}
	The uniform boundedness of $\|v_n\|_\infty$, the convergence of $\{\lambda_n\}$, $\{\eta_n\}$, and the assumptions \ref{WM}, \ref{WA}, \ref{WF1} guarantee the existence of $M>0$ such that $\|\tilde{g}_n\|_\gamma \leq M$ for all $n$, where $\gamma>N/p$ if $N\ge p$ and $\gamma=1$ if $N<p$.
	Consequently, 
	we infer from  Proposition~\ref{prop:C0-reg} the existence of $\beta \in (0,1)$ such that for any compact set $K \subset \Omega$ there is $C>0$ such that  $\|v_n\|_{C^{0,\beta}({K})} \leq C$ for all $n$.
	By the Arzel\`a-Ascoli theorem, $\{v_n\}$ converges to $v_0$ in $C({K})$, up to a subsequence. 
	This is the desired $C^0_{\mathrm{loc}}(\Omega)$-convergence.
	Recalling that $v_0$ is nontrivial, we conclude from \eqref{eq:tildegnx} and the strong convergence $v_n \to v_0$ in $\W(\Omega)$ that $v_0$ is an eigenfunction of the problem \eqref{eq:EP} associated with the eigenvalue $\lambda$.
\end{proof}

Under stronger regularity assumptions on the parameters of \eqref{eq:Psub}, we get the following improvement of Lemma~\ref{lem:auxil0}. 
\begin{lemma}\label{lem:auxil1}
	Let \ref{O}, \ref{M}, \ref{A}, \ref{F1} be satisfied.
	Then, in addition to the assertions of Lemma~\ref{lem:auxil0}, the normalized sequence $\{u_n/\|u_n\|_\infty\}$ converges in $C^1(\overline{\Omega})$ to an eigenfunction associated with the eigenvalue $\lambda$, up to a subsequence. 
\end{lemma}
\begin{proof}
	The argument is built upon the proof of Lemma~\ref{lem:auxil0} and complements it. 
	Recall that each function $v_n = u_n/\|u_n\|_\infty$ satisfies the problem
	\eqref{eq:lem:v} with $\tilde{g}_n$ given by \eqref{eq:tildegnx}.
	Thanks to the regularity assumptions \ref{M}, \ref{A}, \ref{F1} imposed on $m$, $a$, $f$, there exist $\gamma>N$ and ${M}_1>0$ such that $\|\tilde{g}_n\|_\gamma \leq M$ for all $n$.
	Consequently, recalling \ref{O},  we infer from Proposition~\ref{prop:C1-reg} the existence of $\tilde{\beta} \in (0,1)$ and $C>0$ such that  $\|v_n\|_{C^{1,\tilde{\beta}}(\overline{\Omega})} \leq C$ for all $n$. 
	By the Arzel\`a-Ascoli theorem, $\{v_n\}$ converges in $C^1(\overline{\Omega})$, up to a subsequence, to an eigenfunction of \eqref{eq:EP} associated with the eigenvalue $\lambda=\lim_{n \to \infty} \lambda_n$.	
\end{proof}

The following two ``bifurcation from infinity''-type lemmas are crucial for the proofs of our main results on the MP and AMP.
They show that the assumption of Lemma~\ref{lem:auxil0} (and hence of Lemma~\ref{lem:auxil1}) on the divergence of solutions in $\W(\Omega)$ is satisfied if $\lambda$ approaches $\lambda_1(m)$ and $\eta$ approaches $0$.
\begin{lemma}\label{lem:conver0}
	Let \ref{WM}, \ref{WA}, \ref{WF1}, \ref{F2} be satisfied. 
	Let $\{\lambda_n\}, \{\eta_n\} \subset \mathbb{R}$ be arbitrary sequences such that
	\begin{equation}\label{eq:sec:le0}
		\lim_{n \to\infty}\lambda_n=\lambda_1(m) 
		\quad \text{and} \quad 
		\lim_{n\to\infty}\eta_n=0. 
	\end{equation}
	Let $u_n \in \W(\Omega)$ be a solution of {\renewcommand{\Plef}{\mathcal{P};\lambda_n,\eta_n}\eqref{eq:Psub}}.
	Then $\|\nabla u_n\|_p \to \infty$, $\|u_n\|_\infty \to \infty$, and there exists $t \neq 0$ such that $\{u_n/\|u_n\|_\infty\}$ converges to $t\varphi_1$ in $\W(\Omega)$ and $C^0_{\mathrm{loc}}(\Omega)$, up to a subsequence.
\end{lemma}
\begin{proof}
	Let us show that $\|\nabla u_n\|_p \to \infty$. 
	Suppose, by contradiction, that the sequence $\{\|\nabla u_n\|_p\}$ is bounded. 
	We see from Lemma~\ref{lem:convsol} that $\{u_n\}$ converges in $\W(\Omega)$ to a solution of the problem~\eqref{eq:EPinhom}, up to a subsequence.
	However, this is impossible in view of the assumption~\ref{F2}, and hence $\|\nabla u_n\|_p \to \infty$.
	The remaining results follow from Lemma~\ref{lem:auxil0} together with the simplicity of $\lambda_1(m)$. 
\end{proof}

\begin{lemma}\label{lem:conver}
	Let \ref{O}, \ref{M}, \ref{BM}, \ref{A}, \ref{F1}, \ref{F2} be satisfied.
	Then, in addition to the assertions of Lemma~\ref{lem:conver0}, the normalized sequence $\{u_n/\|u_n\|_\infty\}$ converges either to $\varphi_1$ or to $-\varphi_1$ in $C^1(\overline{\Omega})$, up to a subsequence.
	In particular, for all sufficiently large $n$, we have either $u_n>0$ in $\Omega$ and $\partial u_n/\partial \nu<0$ on $\partial\Omega$, or $u_n<0$ in $\Omega$ and $\partial u_n/\partial \nu>0$ on $\partial\Omega$, up to a subsequence.
\end{lemma}
\begin{proof}
	Since we know that $\|\nabla u_n\|_p \to \infty$ from Lemma~\ref{lem:conver0}, we apply Lemma~\ref{lem:auxil1} and recall the normalization assumption $\|\varphi_{1}\|_\infty=1$ to conclude that the sequence consisting of normalized functions $v_n = u_n/\|u_n\|_\infty$ converges either to $\varphi_1$ or to $-\varphi_1$ in $C^1(\overline{\Omega})$, up to a subsequence.
	Since $\varphi_1>0$ in $\Omega$ and, thanks to the assumption \ref{BM}, we have $\partial \varphi_1/\partial \nu<0$ on $\partial\Omega$,  the convergence $v_n \to \varphi_1$ in $C^1(\overline{\Omega})$ implies $u_n>0$ in $\Omega$ and $\partial u_n/\partial \nu<0$ on $\partial \Omega$ for all sufficiently large $n$, 
	and the converse inequalities hold true in the case of the convergence $v_n \to -\varphi_1$ in $C^1(\overline{\Omega})$.
\end{proof}

\subsection{Properties of \texorpdfstring{$\eta^*_\lambda(a)$}{eta*-lambda(a)}}

Let us discuss some properties of the critical value $\eta^*_\lambda(a)$ defined in \eqref{eq:eta*1}.
\begin{lemma}\label{lem:eta}
	Let \ref{WM}, \ref{WA}, \ref{WF1} be satisfied, and $u \in \W(\Omega) \setminus \{0\}$.
	Assume that $u, f \geq 0$ a.e.\ in $\Omega$.
	Let either of the following assumptions hold:
	\begin{enumerate}[label={{\rm {(\roman*)}}}]	
	\addtolength{\itemindent}{0em}
	\item\label{lem:eta:1} 
	$0 \leq \lambda \leq \lambda_1(m)$ and 
	$\eta \intO a u^q \,dx \leq 0$;
	\item\label{lem:eta:2}
	$0 \leq \lambda < \lambda_1(m)$ and either 
	$-\eta^*_\lambda(-a) < \eta \leq 0$ 
	or 
	$0 \leq \eta < \eta^*_\lambda(a)$;
	\item\label{lem:eta:3}
	$\lambda = \lambda_1(m)$, either 
	$-\eta^*_\lambda(-a) < \eta \leq 0$ 
	or 
	$0 \leq \eta < \eta^*_\lambda(a)$, and $u \neq t\varphi_{1}$ for any $t>0$.
	\end{enumerate}	
	Then $u$ satisfies
	\begin{equation}\label{eq:lem-eta1}
		\intO |\nabla u|^p \,dx 
		-
		\lambda \intO mu^p\,dx 
		- 
		\eta \intO a u^q\,dx 		
		+
		\intO fu\,dx > 0.
	\end{equation}
\end{lemma}
\begin{proof}
	Let $u \in \W(\Omega) \setminus \{0\}$ be any nonnegative function.
	Observe that, regardless the sign of $\intO mu^p\,dx$, we have
	\begin{equation}\label{eq:lem:ine1}
		\intO |\nabla u|^p \,dx -	\lambda \intO mu^p\,dx \geq 0
	\end{equation}
	provided $0 \leq \lambda \leq \lambda_1(m)$, and the
	equality holds in \eqref{eq:lem:ine1} if and only if $\lambda=\lambda_1(m)$ and $u=t\varphi_1$ for some $t>0$. 
	Thus, 
	recalling that $f \geq 0$ a.e.\ in $\Omega$ and $\varphi_{1}>0$ in $\Omega$, we see that 
	\eqref{eq:lem-eta1} is satisfied under the assumption~\ref{lem:eta:1}.
	
	Consider the case $\eta \intO a u^q \,dx > 0$ and either the assumption~\ref{lem:eta:2} or \ref{lem:eta:3}. 
	We see that the strict inequality holds in \eqref{eq:lem:ine1}.	
	Assume first that $\eta > 0$ and $\intO a u^q \,dx > 0$. 
	Hence, by the assumptions of the lemma on $\eta$, we have $0<\eta < \eta^*_\lambda(a)$,
	while the inequality $\intO a u^q \,dx > 0$ implies that $u \in \Theta(a)$, where $\Theta(a)$ is defined in \eqref{eq:Theta}.	
	Consequently, we get $\intO fu\,dx > 0$, see Remark~\ref{rem:eta}.
	Let us investigate a function $F:[0,\infty) \to \mathbb{R}$ defined as
	$$
	F(t) = 
	t^{p-1} \left(\intO |\nabla u|^p \,dx 
	-
	\lambda \intO mu^p\,dx \right)
	- 
	\eta t^{q-1} \intO a u^q\,dx 		
	+
	\intO fu\,dx.
	$$
	The desired inequality \eqref{eq:lem-eta1} is equivalent to $F(1)>0$. 
	Let us prove that, in fact, $F(t)>0$ for \textit{all} $t \geq 0$. 
	Recalling that $1<q<p$ and that the strict inequality in \eqref{eq:lem:ine1} holds, we see that $F(t)>0$ for any sufficiently small and any sufficiently large $t \geq 0$.
	In particular, $F$ possesses a global minimum point $t_0 \geq 0$.
	Suppose, contrary to our claim, that $F(t_0) \leq 0$. 
	Since $F(0)>0$, we have $t_0>0$ and $F'(t_0) = 0$, and hence
	\begin{equation}\label{eq:lemF0}
		t_0 F'(t_0) = 
		(p-1) t_0^{p-1} \left(\intO |\nabla u|^p \,dx 
		-
		\lambda \intO mu^p\,dx \right)
		- 
		\eta (q-1) t_0^{q-1} \intO a u^q\,dx = 0.
	\end{equation}
	Taking into account the second equality in \eqref{eq:lemF0}, we get
	$$
	F(t_0) =
	-\frac{p-q}{q-1} \, t_0^{p-1} \left(\intO |\nabla u|^p \,dx 
	-
	\lambda \intO mu^p\,dx \right)
	+ \intO fu\,dx
	\leq 0, 
	$$
	and hence
	\begin{equation}\label{eq:t0lem}
		t_0 \geq \left(\frac{q-1}{p-q}\right)^\frac{1}{p-1}
		\left(\frac{\intO fu\,dx}{\intO |\nabla u|^p \,dx 
			-
			\lambda \intO mu^p\,dx}\right)^\frac{1}{p-1}.
	\end{equation}
	Expressing now $\eta$ from \eqref{eq:lemF0} and estimating it from below using \eqref{eq:t0lem}, we obtain the following contradiction:
	\begin{align*}
		\eta^*_\lambda(a) > \eta 
		&= 
		\frac{p-1}{q-1} \frac{\intO |\nabla u|^p \,dx 
			-
			\lambda \intO mu^p\,dx}{\intO a u^q\,dx}
		t_0^{p-q} 
		\\
		&\geq
		\frac{p-1}{(p-q)^\frac{p-q}{p-1}(q-1)^\frac{q-1}{p-1}}
		\frac{\left(\intO |\nabla u|^p \,dx - \lambda \intO m u^p \,dx\right)^\frac{q-1}{p-1} \left(\intO f u\right)^\frac{p-q}{p-1}}{\intO a u^q \,dx}
		\geq \eta^*_\lambda(a).
	\end{align*}
	Assume now that $\eta < 0$ and $\intO a u^q \,dx < 0$. 
	The latter inequality reads as $\intO (-a) u^q \,dx > 0$, and hence $u \in \Theta(-a)$. 
	Repeating the analysis of the function $F$ as above, we derive a contradiction to the assumption $-\eta^*_\lambda(-a)<\eta$. 
	This completes the proof of the inequality~\eqref{eq:lem-eta1}.	
\end{proof}

\begin{lemma}\label{lem:eta-est}
	Let \ref{WM}, \ref{WF1} be satisfied. 
	Assume that $f \geq c$ a.e.\ in $\Omega$ for some $c>0$.
	Assume that $a_+ \in L^r(\Omega) \setminus \{0\}$, where $r>\frac{(p-1)N}{(q-1)p}$ if $N \geq p$ and $r=\frac{p-1}{q-1}$ if $N<p$.
	Then $\eta^*_\lambda(a) > 0$ whenever 
	$0 \leq \lambda < \lambda_1(m)$.	
\end{lemma}
\begin{proof}
	Since $a_+$ is nontrivial, it is not hard to see that $\Theta(a) \neq \emptyset$, where $\Theta(a)$ is defined in \eqref{eq:Theta}.
	Take any $u \in \Theta(a)$.
	Under the imposed assumptions, we
	use the H\"older inequality and the definition \eqref{eq:lambda1} of $\lambda_1\Big(a_+^\frac{p-1}{q-1}\Big)$ (i.e., $\lambda_1(m)$ with $m = a_+^\frac{p-1}{q-1}$ which satisfies \ref{WM} in view of the imposed integrability assumptions on $a_+$) to get
	\begin{align}
		\notag
		0<\intO a u^q \,dx 
		&\leq 
		\intO a_+ u^q \,dx 
		\leq
		\left(
		\intO a_+^\frac{p-1}{q-1} u^p \,dx
		\right)^\frac{q-1}{p-1}
		\left(
		\intO u \,dx
		\right)^\frac{p-q}{p-1}
		\\
		\label{eq:lem-eta-1}
		&\leq
		c^{-\frac{p-q}{p-1}}
		\lambda_1\Big(a_+^\frac{p-1}{q-1}\Big)^{-\frac{q-1}{p-1}}
		\left(
		\intO |\nabla u|^p \,dx
		\right)^\frac{q-1}{p-1}
		\left(
		\intO f u \,dx
		\right)^\frac{p-q}{p-1}.
	\end{align}
	On the other hand, since $0 \leq \lambda < \lambda_1(m)$, we obtain
	\begin{equation}\label{eq:lem-eta-2}
		\intO |\nabla u|^p \,dx - \lambda \intO m u^p \,dx
		\geq
		\left(1-\frac{\lambda}{\lambda_1(m)}\right)
		\intO |\nabla u|^p \,dx,
	\end{equation}
	where we employed the definition \eqref{eq:lambda1} of $\lambda_1(m)$ in the case $\intO m u^p \,dx > 0$.
	Thus, using \eqref{eq:lem-eta-1} and \eqref{eq:lem-eta-2}, we derive the following explicit lower bound for $\eta_\lambda^*(a)$:
	\begin{equation*}
		\eta_\lambda^* \geq 
		\frac{p-1}{(p-q)^\frac{p-q}{p-1}(q-1)^\frac{q-1}{p-1}}
		\, 
		c^\frac{p-q}{p-1}
		\lambda_1\Big(a_+^\frac{p-1}{q-1}\Big)^\frac{q-1}{p-1}
		\left(1-\frac{\lambda}{\lambda_1(m)}\right)^\frac{q-1}{p-1} > 0.
		\qedhere
	\end{equation*}
\end{proof}

\subsection{Weighted eigenvalue problem}\label{sec:eigen}
In addition to the information on the weighted eigenvalue problem \eqref{eq:EP} provided in Section~\ref{sec:intro}, let us discuss a few other properties of \eqref{eq:EP} which will be used in the proofs of our main results.

Recall that $\sigma(-\Delta_p\,;m)$ stands for the spectrum of \eqref{eq:EP}.
If, in addition to \ref{WM}, $m_-$ is nontrivial, i.e., we are in the so-called indefinite weight case, then
$\lambda\in\sigma(-\Delta_p;m)$ if and only if 
$-\lambda\in\sigma(-\Delta_p;-m)$. 
In particular, $-\lambda_1(-m)$ is also a principal (but negative) eigenvalue of the problem \eqref{eq:EP}.

Let $O$ be an open subset of $\Omega$.
Define
\begin{equation}\label{eq:lambda1O}
	\lambda_1(m;O)
	:=
	\inf
	\left\{
	\frac{\int_{O} |\nabla u|^p\,dx}{\int_{O} m|u|^p\,dx}:\, u\in\W(O),\ 
	\int_{O} m|u|^p\,dx>0
	\right\}
\end{equation}
and put $\lambda_1(m;O) = \infty$ if the admissible set $\{u\in\W(O),\ 
\int_{O} m|u|^p\,dx>0\}$ is empty. 
By definition, we have $\lambda_1(m) \equiv \lambda_1(m;\Omega)$.
There is the following domain monotonicity type property: if $O$ is a proper subset of $\Omega$, then 
$\lambda_1(m;\Omega)<\lambda_1(m;O)$, see \cite[Proposition~4.4]{Cuesta}. 

Recall the notation
$$
\Omega_\rho 
= 
\{
x \in \Omega:\, \mathrm{dist}(x,\partial\Omega) < \rho
\}.
$$
In particular, we have $\lambda_1(m;\Omega)<\lambda_1(m;\Omega_{\rho})$ for any $\rho>0$ such that $\Omega_\rho$ is a proper subset of $\Omega$.
The following simple topological lemma takes place.
\begin{lemma}\label{lem:topology}
	Let $\rho>r>0$ be such that $\Omega \setminus \Omega_{r}$ is nonempty.
	Then any connected component of $\Omega_{\rho}$ intersects with $\Omega \setminus \Omega_{r}$.
\end{lemma}
\begin{proof}
	Let $O$ be any connected component of $\Omega_{\rho}$. 
	Take any $y \in \Omega \setminus \Omega_{r}$ and $z \in O$.
	If $\mathrm{dist}(z,\partial\Omega) \geq r$, then $z \in \Omega \setminus \Omega_{r}$ and we are done.
	Assume that $\mathrm{dist}(z,\partial\Omega) < r \leq \mathrm{dist}(y,\partial\Omega)$.
	Since $\Omega$ is connected, there is a continuous path $\gamma:[0,1] \to \Omega$ such that $\gamma(0) = z$ and $\gamma(1)=y$. 
	It is well known that the distance function is continuous, and so there exists $t_0 \in (0,1]$ such that 
	$\mathrm{dist}(\gamma(t),\partial\Omega) < r$ for any $t \in [0,t_0)$ and 	$\mathrm{dist}(\gamma(t_0),\partial\Omega) = r$.
	Therefore, $\gamma([0,t_0]) \subset O$ and
	$\gamma(t_0) \in \Omega \setminus \Omega_{r}$.
	Since the choice of $O$ is arbitrary, our assertion is proved.
\end{proof}

The previous lemma allows to obtain the following result which will be used in the proofs of 
Theorems~\ref{thm1-w}, \ref{thm-1ww}, and Proposition~\ref{prop:thm1-w}.
\begin{lemma}\label{lem:eigenpos}
	Let \ref{WM} be satisfied and $\lambda \in \mathbb{R}$.
	Assume that there exist $\rho>r>0$ and $u \in \W(\Omega) \cap C(\Omega)$ such that $u \geq 0$ in $\Omega$, $u>0$ in $\Omega \setminus {\Omega_{r}}$, and $u$ satisfies the inequality
	$$
	-\Delta_p u \geq \lambda \,m(x)|u|^{p-2}u 
	\quad \text{in} ~ \Omega_{\rho}
	$$
	in the weak sense. 
	Then $u>0$ in $\Omega$.
\end{lemma}
\begin{proof}
	Applying the weak Harnack inequality given by \cite[Theorem~7.1.2 and a subsequent remark]{PS} when $p \leq N$ (see also \cite[Corollary~7.1.3]{PS}) and \cite[Theorem~7.4.1]{PS} when $p > N$ on every connected component of $\Omega_{\rho}$ and using Lemma~\ref{lem:topology}, we conclude that $u>0$ in $\Omega_{\rho}$ and hence in the whole $\Omega$. 
\end{proof}

The following result will be needed for the proof of Theorem~\ref{thm5}.
Let $\{O_\rho\}$ be a sequence of domains such that each $O_\rho$ is compactly contained in $\Omega$ and $\Omega \setminus O_\rho \subset \Omega_{\rho}$. 
Denote
by $\phi_\rho \in \W(O_\rho)$ a positive eigenfunction corresponding to $\lambda_1(m;O_\rho)$, whenever $\lambda_1(m;O_\rho) \in (0, \infty)$.
We may assume that $\phi_\rho\in \W(\Omega)$ by the zero extension. 
The behavior of $\lambda_1(m;O_\rho)$ and $\phi_\rho$ as $\rho \to 0$ is described in the following lemma.
\begin{lemma}\label{lem:eigenconv}
	Let \ref{WM} be satisfied.
	Then $\lambda_1(m;O_\rho)  \in (\lambda_1(m;\Omega),\infty)$ for any sufficiently small $\rho>0$.
	Moreover,
	$\lambda_1(m;O_\rho) \to \lambda_1(m;\Omega)$ and  
	$\phi_\rho/(\intO m\phi_{\rho}^p\,dx)^{1/p} \to \varphi_1/(\intO m\varphi_1^p\,dx)^{1/p}$ in $\W(\Omega)$ as 
	$\rho\to 0$. 
\end{lemma}
\begin{proof}
	Let $\{\rho_n\}$ be any sequence which converges to $0$.
	By standard methods, using the molifiers and cut-off functions, we can construct
	a sequence $\{\varphi_{1,n}\} \subset C_0^\infty(\Omega)$ such that
	$\mathrm{supp}\,\varphi_{1,n} \subset O_{\rho_n}$ and  $\varphi_{1,n}\to \varphi_1$ in $\W(\Omega)$. 
	Since, by \ref{WM},
	$$
	\int_{O_{\rho_n}} m\varphi_{1,n}^p\,dx
	=\intO m\varphi_{1,n}^p\,dx\to \intO m\varphi_{1}^p\,dx>0
	\quad \text{as}~ n \to \infty,
	$$ 
	the admissible set for the definition \eqref{eq:lambda1O} of $\lambda_1(m; O_{\rho_n})$ is nonempty for any sufficiently large $n$, and hence
	\begin{align*}
		\lambda_1(m;\Omega)
		< 
		\frac{\int_{O_{\rho_n}} |\nabla \phi_{\rho_n}|^p\,dx}
		{\int_{O_{\rho_n}} m \phi_{\rho_n}^p\,dx}
		&=
		\lambda_1(m;O_{\rho_n})
		\le 
		\frac{\int_{O_{\rho_n}} |\nabla \varphi_{1,n}|^p\,dx}{\int_{O_{\rho_n}} m\varphi_{1,n}^p\,dx} 
		\to 
		\frac{\intO |\nabla \varphi_{1}|^p\,dx}
		{\intO m\varphi_{1}^p\,dx}=\lambda_1(m;\Omega)
	\end{align*}
	as $n \to \infty$.
	We see that $\lambda_1(m;O_{\rho_n}) \to \lambda_1(m;\Omega)$ and the simplicity of $\lambda_1(m;\Omega)$ leads to the convergence of the normalized sequence $\{\phi_{\rho_n}/(\intO m\phi_{\rho_n}^p\,dx)^{1/p}\}$ to $\varphi_1/(\intO m\varphi_1^p\,dx)^{1/p}$ in $\W(\Omega)$.
\end{proof}

In the proof of Theorem~\ref{thm:existence} on the existence of solutions of the problem \eqref{eq:Psub}, we will work with the sequence of variational eigenvalues $\{\lambda_k(m)\}$ of \eqref{eq:EP} defined, using the construction from \cite{drabrob1999}, as
\begin{equation}\label{lambda_n}
	\lambda_k(m) = \inf_{h\in\mathscr{F}_k(m)} \max_{z \in S^{k-1}} 
	\|\nabla h(z)\|_p^p, 
	\quad k \in \mathbb{N},
\end{equation}
where $S^{k-1}$ denotes the unit sphere in $\mathbb{R}^k$ and
\begin{align} 
	\label{F_n} 
	\mathscr{F}_k(m)
	&:=\left\{ h\in C(S^{k-1},S(m)):\, h \text{ is odd}\right\},
	\\
	\label{S_m} 
	S(m) 
	&:=\left\{u\in \W(\Omega):\, \intO m|u|^p\,dx=1\,\right\}.
\end{align} 
It is known that each $\lambda_k(m)$ is indeed an eigenvalue of \eqref{eq:EP} and $\lambda_k(m) \to \infty$ as $k \to \infty$, see \cite[Remark 2.1]{Cuesta}, but it is not known whether $\{\lambda_k(m)\}$ exhausts the positive part of $\sigma(-\Delta_p;m)$, except in the cases $N=1$ or $p=2$.

Let us explicitly mention that $\lambda_1(m)<\lambda_2(m)$ and there is no eigenvalue of \eqref{eq:EP} in between them, see~\cite[Proposition~4.2 and Corollary~5.1]{Cuesta}.

Finally, we refer to \cite[Chapter~3]{pyat} for an overview on the weighted eigenvalue problem \eqref{eq:EP} in the linear case $p=2$.

\section{MP and AMP on subsets of \texorpdfstring{$\Omega$}{Omega}}\label{sec:mpampsubsets}
In order to prove Theorem~\ref{thm1-w}, we prepare two results about ``local'' versions of the MP and AMP on compact subsets of $\Omega$, which might be of independent interest. 
We refer to \cite[Theorem~4.2]{pinch} for a related version of the AMP.
\begin{proposition}\label{prop:thm0-w}
	Let
	\ref{WM}, \ref{WA}, \ref{WF1}, \ref{F2} be satisfied. 
	Assume that $\intO a \varphi_1^q \,dx > 0$.
	Then for any compact subset 
	$K \subset \Omega$ there exists $\delta>0$ such that any solution $u$ of \eqref{eq:Psub} satisfies $u>0$ in $K$  provided $\lambda \in (\lambda_1(m)-\delta,\lambda_1(m))$ and $\eta \in (-\delta,0]$.
\end{proposition}
\begin{proof}
	Suppose, by contradiction, that there exist a compact subset $K$ of $\Omega$, sequences $\lambda_n \nearrow \lambda_1(m)$ and  $\eta_n \nearrow 0$ (the case $\eta_n = 0$ is permitted), and a sequence $\{u_n\}$ of solutions of {\renewcommand{\Plef}{\mathcal{P};\lambda_n,\eta_n}\eqref{eq:Psub}} such that $\min_K u_n \leq 0$ for all $n$.
	Recall that each $u_n \in C(\Omega)$ by Proposition~\ref{prop:C0-reg}, which implies that the minimum is attained.
	We deduce from Lemma~\ref{lem:conver0} that $\|\nabla u_n\|_p \to \infty$, $\|u_n\|_\infty \to \infty$, and, since $\varphi_{1}>0$ in $\Omega$, $\{u_n/\|u_n\|_\infty\}$ converges to $-t\varphi_1$ in $\W(\Omega)$ and $C(K)$ for some $t>0$, up to a subsequence.
	
	Taking $-u_n/\|u_n\|_\infty^p$ as a test function for {\renewcommand{\Plef}{\mathcal{P};\lambda_n,\eta_n}\eqref{eq:Psub}} and denoting $v_n = -u_n/\|u_n\|_\infty$, we get
	\begin{equation}\label{eq:xi10}
		\intO |\nabla v_n|^{p} \,dx
		=
		\lambda_n \intO m 
		|v_n|^p \, dx 
		+ 
		\frac{\eta_n}{\|u_n\|_\infty^{p-q}} \intO a |v_n|^q \,dx
		-
		\frac{1}{\|u_n\|_\infty^{p-1}}\intO f v_n\,dx.
	\end{equation}
	The convergence $v_n \to t\varphi_{1}$ in $\W(\Omega)$ and the regularity assumptions \ref{WM}, \ref{WA}, \ref{WF1} give the convergences
	\begin{align*}
		&\intO m |v_n|^p \, dx 
		\to
		t^p \intO m	\varphi_1^p \, dx > 0,
		\quad
		\intO a |v_n|^q \, dx
		\to
		t^q \intO a \varphi_1^q \, dx > 0,
		\\
		&\intO f v_n\,dx
		\to
		t \intO f \varphi_1 \,dx > 0
		\quad \text{as}~ n \to \infty.
	\end{align*}
	Thus, using the definition \eqref{eq:lambda1} of $\lambda_1(m)$ and recalling that $\eta_n\leq 0$, we obtain from \eqref{eq:xi10} that
	$$
	0 < \lambda_1(m)
	\intO m 
	|v_n|^p \, dx
	\leq 
	\intO |\nabla v_n|^{p} \,dx
	<
	\lambda_n \intO m 
	|v_n|^p \, dx
	$$
	for any sufficiently large $n$, which is impossible since $\lambda_n < \lambda_1(m)$.
\end{proof}

Recall the notation
	$$
	\Omega_{\rho}
	=
	\{x\in\Omega:\,\mathrm{dist}(x,\partial\Omega)<\rho\}.
	$$
\begin{proposition}\label{prop:thm1-w}
	Let
	\ref{WM}, \ref{WA}, \ref{WF1}, \ref{F2} be satisfied. 
	Assume that $\intO a \varphi_1^q \,dx > 0$ and there exists $\rho>0$ such that 
	$a = 0$ a.e.\ in $\Omega_{\rho}$ and $f \geq 0$ a.e.\ in $\Omega_{\rho}$.
	Then for any compact 
	subset $K \subset \Omega$ there exists $\delta>0$ such that any solution $u$ of \eqref{eq:Psub} satisfies $u<0$ in $K$  provided $\lambda \in (\lambda_1(m),\lambda_1(m)+\delta)$ and $\eta \in [0,\delta)$.
\end{proposition}
\begin{proof}
	Suppose, contrary to our claim, that there exist a compact subset $K \subset \Omega$, sequences $\lambda_n \searrow \lambda_1(m)$ and $\eta_n \searrow 0$ (the case $\eta_n=0$ is permitted), and a sequence $\{u_n\}$ of solutions of {\renewcommand{\Plef}{\mathcal{P};\lambda_n,\eta_n}\eqref{eq:Psub}} such that $\max_K u_n \geq 0$ for all $n$.
	Proposition~\ref{prop:C0-reg} guarantees that $u_n \in C(\Omega)$. In particular, this implies that the maximum is attained.
	We deduce from Lemma~\ref{lem:conver0} that $\|\nabla u_n\|_p \to \infty$, $\|u_n\|_\infty \to \infty$, and, since $\varphi_{1}>0$ in $\Omega$, $\{u_n/\|u_n\|_\infty\}$ converges to $t\varphi_1$ in $\W(\Omega)$ (and hence a.e.\ in $\Omega)$ and in $C^0_{\mathrm{loc}}(\Omega)$ for some $t>0$, up to a subsequence.
	Consequently, for any $r \in (0,\rho)$ there exists a constant $c>0$ such that 
	$u_n \geq c$ in $\Omega \setminus \Omega_{r}$ for any sufficiently large $n$, where $\rho>0$ is given by the assumption of the proposition.
	Moreover, recalling that $\lambda_1(m;\Omega)<\lambda_1(m;\Omega_{r})$ (see Section~\ref{sec:eigen}), we can take $n$ larger to guarantee that $\lambda_n \in (\lambda_1(m;\Omega),\lambda_1(m;\Omega_{r}))$.

	Let us prove that $u_n > 0$ in the whole $\Omega$.
	Suppose first that $(u_n)_-$ is nontrivial.
	Since $(u_n)_- \in \W(\Omega) \setminus \{0\}$ and $\mathrm{supp}\,(u_n)_- \subset \overline{\Omega_{r}}$, we use $-(u_n)_-$ as a test function for {\renewcommand{\Plef}{\mathcal{P};\lambda_n,\eta_n}\eqref{eq:Psub}} and, noting that $a = 0$ a.e.\ in $\Omega_{r}$ and 
	$f \geq 0$ a.e.\ in $\Omega_{r}$, we get
	\begin{align}
		0<\int_{\Omega_r} |\nabla (u_n)_-|^p \,dx
		=
		\lambda_n \int_{\Omega_r} m (u_n)_-^p \,dx
		+
		\eta_n \int_{\Omega_r} a (u_n)_-^q \, dx
		&-
		\int_{\Omega_r} f (u_n)_- \,dx
		\\
		\label{eq:u-m}
		&\leq
		\lambda_n \int_{\Omega_r} m (u_n)_-^p \,dx.
	\end{align}
	Since $(u_n)_- \in \W(\Omega) \cap C(\Omega)$ and $(u_n)_-=0$ on $\partial \Omega_r \cap \Omega$,  \cite[Lemma~5.6]{CuestaFucik} ensures that
	$(u_n)_- \in \W(\Omega_{r})$.
	As a consequence, we conclude from \eqref{eq:u-m} that $(u_n)_-$ is admissible for the definition \eqref{eq:lambda1O} of $\lambda_1(m;\Omega_{r})$, which gives the following contradiction:
	\begin{align*}
	0<\lambda_1(m;\Omega_{r})
	\int_{\Omega_{r}} m (u_n)_-^p \,dx 
	\leq 
	\int_{\Omega_{r}} |\nabla (u_n)_-|^p \,dx
	&\leq
	\lambda_n \int_{\Omega_{r}} m (u_n)_-^p \,dx
	\\
	&<
	\lambda_1(m;\Omega_{r})
	\int_{\Omega_{r}} m (u_n)_-^p \,dx.
	\end{align*}
	Thus, $u_n$ is nonnegative in $\Omega$.	
	In view of the inequality $u_n > 0$ in $\Omega \setminus \Omega_{r}$
	and the assumptions $a = 0$ a.e.\ in $\Omega_{\rho}$ and 
	$f \geq 0$ a.e.\ in $\Omega_{\rho}$, Lemma~\ref{lem:eigenpos} implies that $u_n > 0$ in the whole $\Omega$. 	
	
	Finally, let us obtain a contradiction to the positivity of $u_n$.
	We know from Lemma~\ref{lem:testfunc} that $\varphi_1^p/(u_n+\varepsilon)^{p-1} \in \W(\Omega)$ for any $\varepsilon>0$, i.e., it is a legitimate test function for the problem {\renewcommand{\Plef}{\mathcal{P};\lambda_n,\eta_n}\eqref{eq:Psub}}.
	Thus, applying the Picone inequality given by Lemma~\ref{lem:picone-weak}, we obtain
	\begin{align}
		\notag
		&\lambda_1(m) \intO m 
		\varphi_1^p \, dx  
		=
		\intO |\nabla \varphi_1|^{p} \,dx
		\geq
		\intO |\nabla u_n|^{p-2} \nabla u_n\nabla \left(\frac{\varphi_1^p}{(u_n+\varepsilon)^{p-1}}\right) dx\\
		\label{eq:prop:weak1}
		&=\lambda_n 
		\intO m \,
		\frac{u_n^{p-1}}{(u_n+\varepsilon)^{p-1}}\,
		\varphi_1^p\, dx 
		+ 
		\eta_n \intO a\, \frac{u_n^{q-1}}{(u_n+\varepsilon)^{p-1}} \,
		\varphi_1^p \, dx
		+
		\intO f
		\frac{\varphi_1^{p}}{(u_n+\varepsilon)^{p-1}} \,dx.
	\end{align}
	Now, for a fixed $n$, we let $\varepsilon \searrow 0$.
	Using the dominated convergence theorem, we get
	\begin{align*}
		\intO m \,
		\frac{u_n^{p-1}}{(u_n+\varepsilon)^{p-1}} \,
		\varphi_1^p\, dx
		\to
		\intO m 
		\varphi_1^p\, dx
		> 0
		\quad \text{as}~ \varepsilon \searrow 0.
	\end{align*}
Since $\intO f \varphi_1 \,dx > 0$, we have $\int_{\Omega \setminus \Omega_r} f \varphi_1 \,dx > 0$ for any sufficiently small $r>0$.
Taking any such $r$ and noting that $a = 0$ a.e.\ in $\Omega_{r}$, 
	$f \geq 0$ a.e.\ in $\Omega_{r}$, and $u_n \geq c$ in $\Omega \setminus \Omega_{r}$ for some $c>0$, the dominated convergence theorem also gives
	\begin{align*}
		\intO a\, \frac{u_n^{q-1}}{(u_n+\varepsilon)^{p-1}} \,
		\varphi_1^p \, dx
		=
		\int_{\Omega \setminus \Omega_{r}} a\, \frac{u_n^{q-1}}{(u_n+\varepsilon)^{p-1}} \,
		\varphi_1^p \, dx
		&\to 
		\int_{\Omega \setminus \Omega_{r}} a \, \frac{\varphi_{1}^q}{u_n^{p-q}} \,dx,
		\\
		\intO f
		\frac{\varphi_1^{p}}{(u_n+\varepsilon)^{p-1}} \,dx
		\geq
		\int_{\Omega \setminus \Omega_{r}} f
		\frac{\varphi_1^{p}}{(u_n+\varepsilon)^{p-1}} \,dx
		&\to 
		\int_{\Omega \setminus \Omega_{r}} f \, \frac{\varphi_1^p}{u_n^{p-1}} \,dx
	\end{align*}
 as $\varepsilon \searrow 0$.
	Therefore, passing to the normalized functions $v_n = u_n/\|u_n\|_\infty$, we deduce from \eqref{eq:prop:weak1} that
	\begin{equation}\label{eq:l1-l-1}
		(\lambda_1(m) - \lambda_n)
		\intO m 
		\varphi_1^p \, dx  
		\geq
		\frac{\eta_n}{\|u_n\|_\infty^{p-q}} \int_{\Omega \setminus \Omega_{r}} a\, \frac{\varphi_1^{p-q}}{v_n^{p-q}} \, \varphi_1^q
		\,dx
		+
		\frac{1}{\|u_n\|_\infty^{p-1}}
		\int_{\Omega \setminus \Omega_{r}} f \, \frac{\varphi_1^{p-1}}{v_n^{p-1}} \, \varphi_1\,dx.
	\end{equation}
	Since $v_n \to t\varphi_1$ in $C(\Omega \setminus \Omega_{r})$, we have
	\begin{align*}
		\int_{\Omega \setminus {\Omega_{r}}} a\, \frac{\varphi_1^{p-q}}{v_n^{p-q}} \, \varphi_1^q
		\,dx 
		&\to 
		\frac{1}{t^{p-q}}\int_{\Omega \setminus {\Omega_{r}}} a \varphi_1^q \,dx  
		=
		\frac{1}{t^{p-q}}\intO a \varphi_1^q \,dx
		> 0,\\
		\int_{\Omega \setminus {\Omega_{r}}} f \,\frac{\varphi_1^{p-1}}{v_n^{p-1}} \, \varphi_1\,dx
		&\to 
		\frac{1}{t^{p-1}}\int_{\Omega \setminus {\Omega_{r}}} f \varphi_1\,dx 
		>  0
		\quad \text{as}~n \to \infty, 
	\end{align*}
	thanks to the choice of $r>0$.
	Consequently, recalling that $\eta_n \geq 0$, we deduce from \eqref{eq:l1-l-1} that $\lambda_n < \lambda_1(m)$ for any sufficiently large $n$, which contradicts our assumption $\lambda_n > \lambda_1(m)$.
\end{proof}

For convenience, we separately state the results of Propositions~\ref{prop:thm0-w} and~\ref{prop:thm1-w} in the unperturbed case $\eta=0$.
\begin{corollary}\label{cor:AMP-loc}
	Let
	\ref{WM}, \ref{WF1}, \ref{F2} be satisfied. 
	Let $K \subset \Omega$ be a compact set. 
	Then the following 
	assertions hold:
	\begin{enumerate}[label={\rm(\roman*)}]
	\item\label{cor:AMP-loc:1}
		There exists $\delta>0$ such that any solution $u$ of \eqref{eq:Pfred} satisfies $u>0$ in $K$  provided $\lambda \in (\lambda_1(m)-\delta,\lambda_1(m))$.
	\item\label{cor:AMP-loc:2}
		Assume that $f \geq 0$ a.e.\ in $\Omega_{\rho}$ for some $\rho>0$.
		Then there exists $\delta>0$ such that any solution $u$ of \eqref{eq:Pfred} satisfies $u<0$ in $K$  provided $\lambda \in (\lambda_1(m),\lambda_1(m)+\delta)$.
	\end{enumerate}
\end{corollary}

\section{Proofs of qualitative properties of solutions}\label{sec:proofs}

\subsection{MP and AMP}
We start with  the maximum principles given by Theorems~\ref{thm0} and~\ref{thm1-w} \ref{thm1-w:1}.
\begin{proof*}{Theorem~\ref{thm0}}\label{page:thm0:proof}
	Suppose, by contradiction, that there exist sequences  $\lambda_n \nearrow \lambda_1(m)$ and $\eta_n \nearrow 0$ (the case $\eta_n=0$ is permitted) such that each {\renewcommand{\Plef}{\mathcal{P};\lambda_n,\eta_n}\eqref{eq:Psub}} possesses a
	solution $u_n$ violating either $u_n>0$ in $\Omega$ or $\partial u_n/\partial \nu < 0$ on $\partial \Omega$. 
	In view of Lemma~\ref{lem:conver}, we have $u_n < 0$ in $\Omega$ and $\partial u_n/\partial \nu > 0$ on $\partial \Omega$ for all sufficiently large $n$, and 
	$\{u_n/\|u_n\|_\infty\}$ converges to $-\varphi_1$ in $C^1(\overline{\Omega})$, up to a subsequence.
	Taking $-u_n/\|u_n\|_\infty^p$ as a test function for {\renewcommand{\Plef}{\mathcal{P};\lambda_n,\eta_n}\eqref{eq:Psub}} and denoting $v_n = -u_n/\|u_n\|_\infty$, we get
	\begin{equation}\label{eq:xi1}
		\intO |\nabla v_n|^{p} \,dx
		=
		\lambda_n \intO m
		v_n^p \, dx 
		+ 
		\frac{\eta_n}{\|u_n\|_\infty^{p-q}} \intO a v_n^q \,dx
		-
		\frac{1}{\|u_n\|_\infty^{p-1}}\intO f v_n\,dx.
	\end{equation}
	The convergence $v_n \to \varphi_1$ in $C^1(\overline{\Omega})$ yields
	$$
	\intO m 
	v_n^p \, dx
	\to 
	\intO m 
	\varphi_1^p \, dx  > 0,
	\quad
	\intO a v_n^q \,dx
	\to
	\intO a \varphi_1^q \,dx> 0,
	\quad
	\intO f v_n\,dx
	\to 
	\intO f \varphi_1\,dx > 0
	$$
	as $n \to \infty$.
	Recalling that $\eta_n \leq 0$ and using the definition \eqref{eq:lambda1} of $\lambda_1(m)$, we obtain from \eqref{eq:xi1} that
	$$
	0<\lambda_1(m) \intO m v_n^p \, dx \leq 
	\intO |\nabla v_n|^{p} \,dx
	< \lambda_n \intO m 
	v_n^p \, dx 
	$$
	for all sufficiently large $n$, which contradicts our assumption $\lambda_n < \lambda_1(m)$. 	
\end{proof*}

\begin{proof*}{Theorem~\ref{thm1-w} \ref{thm1-w:1}}\label{page:thm-1w:1:proof}
	Let $\rho>0$ be such that $a = 0$ a.e.\ in $\Omega_\rho$ and $f \geq 0$ a.e.\ in $\Omega_\rho$.
	By Proposition~\ref{prop:thm0-w}, fixing any $r \in (0,\rho)$, we can find $\delta>0$ such that any solution of \eqref{eq:Psub} is positive in $\Omega \setminus {\Omega_{r}}$ $(\supset \Omega \setminus {\Omega_{\rho}})$ provided $\lambda \in (\lambda_1(m)-\delta,\lambda_1(m))$ and $\eta \in (-\delta,0]$.
	Let $u$ be any such solution.
	Let us show that $u>0$ in the whole $\Omega$.
	Suppose first, by contradiction, that $u_- \not\equiv 0$ in $\Omega$.
	Since $u>0$ in $\Omega \setminus {\Omega_{r}}$, we have $\mathrm{supp}\,u_- \subset \overline{\Omega_r}$ and hence,  
	using $-u_- \in \W(\Omega)$ as a test function for \eqref{eq:Psub} and noting that $a = 0$ a.e.\ in $\Omega_r$ and $f \geq 0$ a.e.\ in $\Omega_r$, we obtain
	$$
	0<\intO |\nabla u_-|^p \,dx
	=
	\lambda \intO m u_-^p \,dx
	+
	\eta \intO a u_-^q \, dx
	-
	\intO f u_- \,dx
	\leq 
	\lambda \intO m u_-^p \,dx.
	$$
	However, this contradicts the definition~\eqref{eq:lambda1} of $\lambda_1(m)$ since $\lambda<\lambda_1(m)$.
	That is, $u_- = 0$ in $\Omega$. 
	Finally, Lemma~\ref{lem:eigenpos} guarantees that $u>0$ in $\Omega$.
\end{proof*}

Now we prove the antimaximum principles stated in Theorems~\ref{thm1} and \ref{thm1-w}~\ref{thm1-w:2}.

\begin{proof*}{Theorem~\ref{thm1}}\label{page:thm1:proof}
	First, we consider the assumption~\ref{thm1:I}. 
	The arguments are essentially reminiscent of the final part of the proof of Proposition~\ref{prop:thm1-w}, but they are simpler due to the additional regularity assumptions.
	We provide details for the sake of clarity. 
	
	Suppose, contrary to our claim, that there exist sequences  $\lambda_n \searrow \lambda_1(m)$ and $\eta_n \searrow 0$ (the case $\eta_n=0$ is permitted) such that each {\renewcommand{\Plef}{\mathcal{P};\lambda_n,\eta_n}\eqref{eq:Psub}} possesses a
	solution $u_n$ violating either $u_n<0$ in $\Omega$ or $\partial u_n/\partial \nu > 0$ on $\partial \Omega$. 
	Therefore, in view of Lemma~\ref{lem:conver}, we have $u_n > 0$ in $\Omega$ and $\partial u_n/\partial \nu < 0$ on $\partial \Omega$ for all sufficiently large $n$, and 
	$\{u_n/\|u_n\|_\infty\}$ converges to $\varphi_1$ in $C^1(\overline{\Omega})$, up to a subsequence.
	Since $u_n, \varphi_1 \in C^1(\overline{\Omega})$, we have
	$$
	\nabla \left(\frac{\varphi_1^p}{u_n^{p-1}}\right)
	=
	p \,\frac{\varphi_1^{p-1}}{u_n^{p-1}}
	\,\nabla \varphi_1 
	-
	(p-1) \,\frac{\varphi_1^{p}}{u_n^{p}}\, \nabla u_n
	\quad \text{in}~ \Omega.
	$$
	Noting that $\varphi_{1}/u_n \in L^\infty(\Omega$), we deduce that
	$\varphi_1^p/u_n^{p-1} \in \W(\Omega)$. 
	Taking $\varphi_1^p/u_n^{p-1}$ as a test function for {\renewcommand{\Plef}{\mathcal{P};\lambda_n,\eta_n}\eqref{eq:Psub}}
	and applying the Picone inequality \cite[Theorem~1.1]{Alleg}, we get
	\begin{align*}
		0<\lambda_1(m) \intO m 
		\varphi_1^p \, dx  
		&=
		\intO |\nabla \varphi_1|^{p} \,dx
		\geq
		\intO |\nabla u_n|^{p-2} \nabla u_n\nabla \left(\frac{\varphi_1^p}{u_n^{p-1}}\right) dx\\
		&=\lambda_n \intO m
		\varphi_1^p \, dx 
		+ 
		\eta_n \intO a \,\frac{\varphi_1^{p-q}}{u_n^{p-q}}
		\, \varphi_1^q \,dx
		+
		\intO f
		\,\frac{\varphi_1^{p-1}}{u_n^{p-1}}\, \varphi_1\,dx.
	\end{align*}
	Passing to the normalized functions $v_n = u_n/\|u_n\|_\infty$, we obtain
	\begin{equation}\label{eq:l1-l}
		(\lambda_1(m) - \lambda_n)
		\intO m 
		\varphi_1^p \, dx  
		\geq
		\frac{\eta_n}{\|u_n\|_\infty^{p-q}} \intO a \,\frac{\varphi_1^{p-q}}{v_n^{p-q}}\, \varphi_1^q
		\,dx
		+
		\frac{1}{\|u_n\|_\infty^{p-1}}
		\intO f \,\frac{\varphi_1^{p-1}}{v_n^{p-1}}\, \varphi_1\,dx.
	\end{equation}
	The convergence $v_n \to \varphi_1$ in $C^1(\overline{\Omega})$ yields
	$$
	\intO a \,\frac{\varphi_1^{p-q}}{v_n^{p-q}}\, \varphi_1^q
	\,dx 
	\to 
	\intO a \varphi_1^q \,dx  > 0
	\quad \text{and} \quad 
	\intO f \,\frac{\varphi_1^{p-1}}{v_n^{p-1}}\, \varphi_1\,dx
	\to 
	\intO f \varphi_1\,dx > 0
	$$
	as $n \to \infty$.
	Consequently, recalling that $\eta_n \geq 0$, we deduce from \eqref{eq:l1-l} that $\lambda_n < \lambda_1(m)$, which contradicts our assumption $\lambda_n > \lambda_1(m)$.
	
	Now, we consider the assumption~\ref{thm1:II}. 
	When $\intO a \varphi_1^q \,dx = 0$, it is hard to control the sign of the right-hand side of the inequality \eqref{eq:l1-l}. 
	Nevertheless, under the additional assumption \eqref{eq:Picone-0}, we can use a different test function.
	Namely, arguing by contradiction as above, let us take $\varphi_1^q/u_n^{q-1} \in \W(\Omega)$ as a test function for {\renewcommand{\Plef}{\mathcal{P};\lambda_n,\eta_n}\eqref{eq:Psub}}.
	Recalling that $\intO a\varphi_1^q \,dx=0$, we get
	\begin{equation}\label{eq:thmIii0} 
		\intO |\nabla u_n|^{p-2} \nabla u_n\nabla \left(\frac{\varphi_1^q}{u_n^{q-1}}\right) dx 
		=\lambda_n \intO m 
		u_n^{p-q}\varphi_1^q \, dx 
		+
		\intO f \,\frac{\varphi_1^{q-1}}{u_n^{q-1}}\, \varphi_1\,dx.
	\end{equation}
	Using the generalized Picone inequality \cite[Theorem 1.8]{BT_Picone}, we obtain 
	\begin{equation}\label{eq:thmIii1} 
		\intO |\nabla u_n|^{p-2} \nabla u_n\nabla \left(\frac{\varphi_1^q}{u_n^{q-1}}\right) dx 
		\le 
		\intO |\nabla \varphi_1|^{p-2} \nabla \varphi_1 \nabla 
		\left(\varphi_1^{q-p+1}u_n^{p-q}\right) dx.	
	\end{equation}
	In order to take $\varphi_1^{q-p+1}u_n^{p-q}$ as a test function for the eigenvalue problem \eqref{eq:EP} with $u=\varphi_{1}$ and $\lambda = \lambda_1(m)$ and then simplify  the right-hand side of \eqref{eq:thmIii1}, let us justify that $\varphi_1^{q-p+1}u_n^{p-q} \in \W(\Omega)$.
	As a remark, we observe that \eqref{eq:Picone-0} implies $q-p+1>0$, see \cite[Lemma~1.6]{BT_Picone}.
	We have 
	\begin{equation*}
		\nabla \left(\varphi_1^{q-p+1}u_n^{p-q}\right)
		= 
		(q-p+1) \left(\frac{u_n}{\varphi_1}\right)^{p-q}
		\nabla \varphi_p
		+(p-q)\left(\frac{\varphi_1}{u_n}\right)^{q-p+1}
		\nabla u_n
		\quad \text{in}~ \Omega.
	\end{equation*}
	Thus, recalling that both $u_n$ and $\varphi_1$ satisfy the boundary point lemma,
	we derive that ${u_n}/{\varphi_1}$ and ${\varphi_1}/{u_n}$ are bounded in $\Omega$, which yields
	$\varphi_1^{q-p+1}u_n^{p-q} = \varphi_1 (u_n/\varphi_1)^{p-q}\in \W(\Omega) \cap C(\overline{\Omega})$.
	Using this fact, we obtain from \eqref{eq:thmIii1} and \eqref{eq:EP} the following inequality:
	\begin{equation}\label{eq:frompic}
		\intO |\nabla u_n|^{p-2} \nabla u_n\nabla \left(\frac{\varphi_1^q}{u_n^{q-1}}\right) dx 
		\le
		\lambda_1(m) \intO m
		u_n^{p-q}\varphi_1^q \, dx.
	\end{equation}
	Combining \eqref{eq:frompic} with \eqref{eq:thmIii0} and passing to the normalized functions $v_n = u_n/\|u_n\|_\infty$, we arrive at
	\begin{equation}\label{eq:thmIii2}
		(\lambda_1(m)-\lambda_n)\intO m v_n^{p-q}\varphi_1^q \, dx
		\geq
		\frac{1}{\|u_n\|_\infty^{p-1}}\intO f \,\frac{\varphi_1^{q-1}}{v_n^{q-1}}\, \varphi_1\,dx.
	\end{equation}
	Thanks to the convergence $v_n \to \varphi_1$ in $C^1(\overline{\Omega})$, we have
	\begin{align*}
		&\intO mv_n^{p-q}\varphi_1^q \, dx
		\to 
		\intO m\varphi_1^p \, dx > 0 \quad \text{and} \quad
		\intO f \,\frac{\varphi_1^{q-1}}{v_n^{q-1}}\, \varphi_1\,dx
		\to 
		\intO f\varphi_1 \,dx > 0
	\end{align*}
as $n \to \infty$, 
	and hence \eqref{eq:thmIii2} yields $\lambda_1(m) > \lambda_n$ for all sufficiently large $n$, which is impossible.	
	
	To conclude, we have proved that under either the assumption~\ref{thm1:I} or~\ref{thm1:II} there exists $\delta>0$ such that if $\lambda_1(m)<\lambda<\lambda_1(m)+\delta$ and $0 \leq \eta < \delta$, then any solution $u$ of \eqref{eq:Psub} satisfies $u<0$ in $\Omega$ and $\partial u/\partial \nu > 0$ on $\partial \Omega$.
\end{proof*}

\begin{proof*}{Theorem~\ref{thm1-w}  \ref{thm1-w:2}}\label{page:thm-1w:2:proof}
	Let $\rho>0$ be such that $a, f = 0$ a.e.\ in $\Omega_\rho$.
	By Proposition~\ref{prop:thm1-w}, fixing any $r \in (0,\rho)$, we can find $\delta>0$ such that any solution of \eqref{eq:Psub} is negative in $\Omega \setminus {\Omega_{r}}$ $(\supset \Omega \setminus {\Omega_{\rho}})$ provided $\lambda \in (\lambda_1(m),\lambda_1(m)+\delta)$ and $\eta \in [0,\delta)$.
	Let $u$ be any such solution of \eqref{eq:Psub}. 
	Decreasing $\delta>0$ if necessary, we may assume that $\lambda_1(m)+\delta \leq \lambda_1(m;\Omega_{r})$, see Section~\ref{sec:eigen}.
	
	Let us show that $u<0$ in the whole $\Omega$.
	Suppose first, by contradiction, that $u_+ \not\equiv 0$ in $\Omega$.
	Since $u<0$ in $\Omega \setminus {\Omega_{r}}$, we have $\mathrm{supp}\,u_+ \subset \overline{\Omega_r}$.
	Using $u_+ \in \W(\Omega) \setminus \{0\}$ as a test function for \eqref{eq:Psub} and noting that $a,f = 0$ a.e.\ in $\Omega_r$, we get
	\begin{align*}
		0
		&<
		\int_{\Omega_{r}} |\nabla u_+|^p \,dx
		=
		\intO |\nabla u_+|^p \,dx
		\\
		&=
		\lambda \intO m u_+^p \,dx
		+
		\eta \intO a u_+^q \, dx
		+
		\intO f u_+ \,dx
		=
		\lambda \intO m u_+^p \,dx
		=
		\lambda \int_{\Omega_{r}} m u_+^p \,dx.
	\end{align*}
	Since $u_+ \in \W(\Omega) \cap C(\Omega)$ and $u_+=0$ on $\partial \Omega_r \cap \Omega$,  \cite[Lemma~5.6]{CuestaFucik} ensures that
	$u_+ \in \W(\Omega_{r})$.
	However, this gives the following contradiction to the definition \eqref{eq:lambda1O} of $\lambda_1(m;\Omega_{r})$ and the choice of $\lambda \in (\lambda_1(m),\lambda_1(m)+\delta)$, where $\lambda_1(m)+\delta \leq \lambda_1(m;\Omega_{r})$:
	$$
	0<\lambda_1(m;\Omega_{r})
	\int_{\Omega_{r}} m u_+^p \,dx 
	\leq 
	\int_{\Omega_{r}} |\nabla u_+|^p \,dx
	=
	\lambda \int_{\Omega_{r}} m u_+^p \,dx
	<
	\lambda_1(m;\Omega_{r})
	\int_{\Omega_{r}} m u_+^p \,dx.
	$$
	That is, $u_+ = 0$ in $\Omega$. 
	Applying Lemma~\ref{lem:eigenpos} to $-u$, we deduce that $u<0$ in $\Omega$. 
\end{proof*}

Finally, we establish the versions of the MP and AMP given by Theorems~\ref{thm-1} and \ref{thm-1ww}.
\begin{proof*}{Theorem~\ref{thm-1}}\label{page:thm-1:proof}
	We first prove the assertion \ref{thm-1:2} on the AMP. 
	Let us fix $\delta>0$ as in Corollary~\ref{cor:AMP}.
	Taking $\delta$ smaller if necessary,  we may assume that $\lambda_1(m)+\delta \leq \lambda_2(m)$, where $\lambda_2(m)$ is the second eigenvalue of the problem~\eqref{eq:EP}, see the last part of Section~\ref{sec:eigen}.
	Suppose, contrary to our claim, that there exists $\lambda \in(\lambda_1(m),\lambda_1(m)+\delta)$ and a sequence $\{\eta_n\}$ such that $|\eta_n| \to 0$ and each {\renewcommand{\Plef}{\mathcal{P};\lambda,\eta_n}\eqref{eq:Psub}} possesses a solution $u_n$ which does not satisfy either $u_n < 0$ in $\Omega$ or $\partial u_n/\partial \nu>0$ on $\partial \Omega$. 
	Let us show that $\{\|\nabla u_n\|_p\}$ is bounded. 
	Indeed, if we suppose, by contradiction, that $\|\nabla u_n\|_p \to \infty$ along a subsequence, then Lemma~\ref{lem:auxil1} implies that $\lambda$ is an eigenvalue of the problem \eqref{eq:EP}, which is impossible since $\lambda_1(m) < \lambda < \lambda_2(m)$. 
	This justifies the boundedness of $\{\|\nabla u_n\|_p\}$, and hence 
	Proposition~\ref{prop:bdd} guarantees that  $\{\|u_n\|_\infty\}$ is also bounded. 
	Denoting
	$$
	g_n(x) = \lambda \,m(x)|u_n(x)|^{p-2}u_n(x)
	+ 
	\eta_n \,a(x)|u_n(x)|^{q-2}u_n(x) + f(x),
	\quad x \in \Omega,
	$$
	and recalling the assumptions \ref{M}, \ref{A}, \ref{F1}, we see that $\{g_n\}$ is uniformly bounded in $L^\gamma(\Omega)$ by some constant $M>0$, i.e., $\|g_n\|_\gamma \leq M$ for all $n$, where $\gamma>N$.
	Consequently, in view of \ref{O},  we infer from Proposition~\ref{prop:C1-reg} the existence of $\beta \in (0,1)$ and $C>0$ such that  $\|u_n\|_{C^{1,\beta}(\overline{\Omega})} \leq C$ for all $n$.
	By the Arzel\`a-Ascoli theorem, $\{u_n\}$  converges in $C^1(\overline{\Omega})$ to a solution $u$ of \eqref{eq:Pfred}, up to a subsequence.
	We know from Corollary~\ref{cor:AMP}~\ref{cor:AMP:2} that $u<0$ in $\Omega$ and $\partial u/\partial \nu>0$ on $\partial \Omega$. 
	Thus, by the  $C^1(\overline{\Omega})$-convergence, we deduce that the same inequalities must be preserved for $u_n$ whenever $n$ is large enough. 
	This contradiction completes the proof.
	
	\ref{thm-1:1} The assertion on the MP can be proved arguing in much the same way as above.
	We omit details.
\end{proof*}

\begin{proof*}{Theorem~\ref{thm-1ww}}\label{page:thm-1ww:proof}
	As in the proof of Theorem~\ref{thm-1}, we justify only the assertion~\ref{thm-1ww:2}.
	Let $\rho>0$ be such that $a, f = 0$ a.e.\ in $\Omega_\rho$. 
	Fix some $r \in (0,\rho)$.
	We take $\delta>0$ as in Corollary~\ref{cor:AMP-w} and, 
	decreasing it if necessary,  we may assume that $\lambda_1(m)+\delta \leq \lambda_2(m)$ and $\lambda_1(m)+\delta \leq \lambda_1(m;\Omega_{r})$, see Section~\ref{sec:eigen}.
	Suppose, contrary to our claim, that there exist $\lambda \in(\lambda_1(m),\lambda_1(m)+\delta)$ and a sequence $\{\eta_n\}$ such that $|\eta_n| \to  0$ and each {\renewcommand{\Plef}{\mathcal{P};\lambda,\eta_n}\eqref{eq:Psub}} possesses a solution $u_n$ which does not satisfy $u_n < 0$ in $\Omega$. 
	If $\|\nabla u_n\|_p \to \infty$ along a subsequence, then Lemma~\ref{lem:auxil0} implies that $\lambda$ is an eigenvalue of the problem \eqref{eq:EP}, which is impossible since $\lambda_1(m)<\lambda<\lambda_2(m)$. 
	Therefore, $\{\|\nabla u_n\|_p\}$ is bounded and hence, thanks to Lemma~\ref{lem:convsol}, $\{u_n\}$ converges in $\W(\Omega)$ to a solution $u$ of \eqref{eq:Pfred}, up to a subsequence.
	Moreover, arguing as in the proof of Theorem~\ref{thm-1}~\ref{thm-1:2} but applying Proposition~\ref{prop:C0-reg} instead of Proposition~\ref{prop:C1-reg}, we deduce that $u_n \to u$ in $C^0_{\mathrm{loc}}(\Omega)$, up to a subsequence.
	We know from Corollary~\ref{cor:AMP-w}~\ref{cor:AMP-w:2} that $u<0$ in $\Omega$, which yields $u_n<0$ in $\overline{\Omega \setminus \Omega_{r}}$ 
	for any sufficiently large $n$.
	Let us show that $(u_n)_+ = 0$ in $\Omega$ for such $n$.
	If $(u_n)_+$ is not identically zero, then we have $\mathrm{supp}\,(u_n)_+ \subset \overline{\Omega_r}$. 
	Hence, as in the proof of Theorem~\ref{thm1-w}~\ref{thm1-w:2}, 	
	taking $(u_n)_+$ as a test function for {\renewcommand{\Plef}{\mathcal{P};\lambda,\eta_n}\eqref{eq:Psub}}, we get
	$$
	0<
	\int_{\Omega_{r}} |\nabla (u_n)_+|^p \,dx
	=
	\lambda \int_{\Omega_{r}} m (u_n)_+^p \,dx,
	$$
	and noting that $(u_n)_+ \in \W(\Omega_r)$, we obtain a contradiction to the definition \eqref{eq:lambda1O} of $\lambda_1(m;\Omega_r)$ and the choice of $\lambda$, i.e., $\lambda < \lambda_1(m)+\delta \leq \lambda_1(m;\Omega_r)$. 
	Thus, we conclude from Lemma~\ref{lem:eigenpos} that $u_n<0$ in $\Omega$, opposite to our initial contradictory assumption.
	
	\ref{thm-1ww:1}
	The proof is analogous to that from above, so we omit details.
\end{proof*}

\subsection{Nonuniformity of AMP}

In this subsection, we prove the nonuniformity of the AMP with respect to $f$ stated in Theorems~\ref{thm3} and \ref{thm5}.
\begin{proof*}{Theorem~\ref{thm3}}\label{page:thm3:proof}
	Arguing by contradiction, we assume the existence of $\varepsilon>0$ such that for any nonnegative $f \in C_0^\infty(\Omega)\setminus \{0\}$ one can find
	$\lambda\ge \lambda_1(m)+\varepsilon$ and $\eta\ge 0$ for which
	$\eqref{eq:Psub}$ has either a nonnegative solution or a negative solution. 
	Since $\varphi_1 \in \W(\Omega)$, there exists a sequence $\{\phi_n\} \subset C_0^\infty(\Omega)$ such that $\phi_n \to \varphi_1$ in $\W(\Omega)$. 
	In particular, we have $\intO m |\phi_n|^p\,dx>0$ for all sufficiently large $n$.
	For such $n$, we take any nonnegative $f_n \in C_0^\infty(\Omega) \setminus \{0\}$ satisfying $\mathrm{supp}\,\phi_n \cap \mathrm{supp}\,f_n=\emptyset$. 
	
	Let $u_n \in \W(\Omega)$ be either a nonnegative solution or a negative solution of
	{\renewcommand{\Plef}{\mathcal{P};\lambda_n,\eta_n,f_n}\eqref{eq:Psub}} with some $\lambda_n \geq \lambda_1(m)+\varepsilon$ and $\eta_n \geq 0$.
	Recall that $u_n\in C(\Omega)$ by Proposition~\ref{prop:C0-reg}.
	Notice that the regularity assumptions on $m$, $a$, $f$, and the assumption $a,f_n \geq 0$ a.e.\ in $\Omega$ allow us to apply the weak Harnack inequality given by \cite[Theorem~7.1.2 and a subsequent remark]{PS} when $p \leq N$ (see also \cite[Corollary~7.1.3]{PS}) and \cite[Theorem~7.4.1]{PS} when $p > N$, and hence we deduce that any nonnegative solution of {\renewcommand{\Plef}{\mathcal{P};\lambda_n,\eta_n,f_n}\eqref{eq:Psub}} is actually positive. 	
	Denote $v_n=|u_n|$, so $\min_K v_n > 0$ for every compact subset $K \subset \Omega$. 
	Since $\phi_n \in C_0^\infty(\Omega)$, we deduce from Lemma~\ref{lem:testfunc2} that $|\phi_n|^p/v_n^{p-1} \in \W(\Omega)$ and hence we can use either $|\phi_n|^p/v_n^{p-1}$ or $-|\phi_n|^p/v_n^{p-1}$ as a test function for {\renewcommand{\Plef}{\mathcal{P};\lambda_n,\eta_n,f_n}\eqref{eq:Psub}} in the case of $u_n>0$ or $u_n<0$ in $\Omega$, respectively.
	Applying the Picone inequality given by Lemma~\ref{lem:picone-weak} and recalling that 
	$f_n$, $\phi_n$ have disjoint supports, 
	we deduce that
	$$
	\intO |\nabla \phi_n|^p\,dx
	\geq
	\lambda_n \intO m |\phi_n|^p\,dx
	+
	\eta_n \intO a v_n^{q-p} |\phi_n|^p\,dx 
	\geq 
	\lambda_n \intO m |\phi_n|^p\,dx,
	$$
	where the second inequality is satisfied since  $\eta_n \geq 0$ and $a \geq 0$ a.e.\ in $\Omega$. 
	Recalling that $\phi_n \to \varphi_1$ in $\W(\Omega)$, we obtain 
	$$
	\lambda_1(m) = \frac{\intO |\nabla \varphi_1|^p\,dx}{\intO m \varphi_1^p\,dx}
	=
	\lim\limits_{n \to \infty}
	\frac{\intO |\nabla \phi_n|^p\,dx}{\intO m |\phi_n|^p\,dx}
	\ge \liminf \limits_{n \to \infty} \lambda_n\ge \lambda_1(m)+\varepsilon, 
	$$
	which is impossible.
\end{proof*}

\begin{proof*}{Theorem~\ref{thm5}}\label{page:thm5:proof}
	Suppose, by contradiction, that there exists $\lambda>\lambda_1(m)$ 
	such that for any nonnegative $f \in C_0^\infty(\Omega)\setminus \{0\}$ one can find $\eta \geq 0$ for which
	$\eqref{eq:Psub}$ has either a positive solution or a negative solution. 
	
	In view of the assumption \ref{O}, \cite[Theorem~1.1]{Schmidt}  guarantees the existence of a sequence of \textit{smooth} domains $\{O_\rho\}$ such that each $O_\rho$ is compactly contained in $\Omega$ and $\Omega \setminus O_\rho \subset \Omega_{\rho}$. (See, e.g., \cite[Eq~(4), p.~117]{maggi} for the validity of the assumption \cite[(1.4)]{Schmidt}.)
	As in Section~\ref{sec:eigen}, we denote by $\phi_\rho \in \W(O_\rho)$ a positive eigenfunction corresponding to $\lambda_1(m;O_{\rho})$, and we may assume that $\phi_\rho \in \W(\Omega)$ by the zero extension. 
	Since each $O_\rho$ is smooth and $m\geq 0$ a.e.\ in $\Omega$, we have $\phi_\rho \in C^1(\overline{O_{\rho}})$ and $\partial \phi_\rho/\partial \nu < 0$ on $\partial O_{\rho}$, see Section~\ref{sec:intro} and Remark~\ref{rem:reg-m}.
	We deduce from Lemma~\ref{lem:eigenconv} that 
	\begin{equation}\label{eq:thm5-4}
		\lambda_1(m;\Omega) 
		< 
		\lambda_1(m;O_{\rho})
		<
		\lambda
		\quad {\rm and}\quad 
		\int_{O_{\rho}} a\phi_\rho^q\,dx>0
	\end{equation} 
	for any sufficiently small $\rho>0$, 
	where the last inequality in \eqref{eq:thm5-4} is guaranteed by the assumption 
	$\intO a\varphi_1^q\,dx>0$. 
	For any such $\rho>0$ we choose a smooth nonnegative function $f$ such that 
	$\mathrm{supp}\,f\subset \Omega \setminus O_\rho$.
	In particular, we have
	$\mathrm{supp}\,f \cap \mathrm{supp}\,\phi_\rho=\emptyset$. 
	By our contradictory assumption, 
	we can find a solution $u$ of \eqref{eq:Psub} 
	for such $f$ and some $\eta\ge 0$ which is either positive or negative in $\Omega$. 
	Denote $v=|u|$, so $v>0$ in $\Omega$, and we have $v \in C^1(\overline{\Omega})$ by Proposition~\ref{prop:C1-reg}.

	Since $\mathrm{supp}\,\phi_\rho=\overline{O_{\rho}}\subset \Omega$ and $\min_{\overline{O_\rho}} v>0$, 
	we can take $\phi_\rho^q/v^{q-1}$ as a test function for \eqref{eq:Psub} and get
	\begin{align} 
		\int_{O_\rho} |\nabla v|^{p-2}\nabla v\,\nabla 
		\left(\frac{\phi_\rho^q}{v^{q-1}}\right)dx 
		=\lambda \int_{O_\rho} mv^{p-q}\phi_\rho^q\,dx 
		+\eta \int_{O_\rho} a\phi_\rho^q \,dx
		\label{eq:thm5-5}
	\end{align} 
by recalling that $\mathrm{supp}\,f \cap \mathrm{supp}\,\phi_\rho=\emptyset$. 
	On the other hand, in view of the assumption \eqref{eq:Picone-01},
	the generalized Picone inequality  \cite[Theorem 1.8]{BT_Picone} guarantees that
	\begin{equation}\label{eq:thm5-6}
		\int_{O_\rho} |\nabla v|^{p-2}\nabla v\,\nabla 
		\left(\frac{\phi_\rho^q}{v^{q-1}}\right)dx 
		\le 
		\int_{O_\rho} |\nabla \phi_\rho|^{p-2}\nabla \phi_\rho\,\nabla 
		\left(\phi_\rho^{q-p+1}v^{p-q}\right)dx.
	\end{equation}
	In order to simplify the right-hand side of \eqref{eq:thm5-6}, let us show that 	$\phi_\rho^{q-p+1}v^{p-q}$ belongs to $W_0^{1,1}(O_\rho)$. 
	Due to the regularity of $\phi_\rho$ and $v$, we have
	\begin{equation}\label{eq:picone-gen1}
		\nabla \left(\phi_\rho^{q-p+1}v^{p-q}\right)
		=(q-p+1)\left(\frac{v}{\phi_\rho}\right)^{p-q}\nabla \phi_\rho 
		+(p-q)\left(\frac{\phi_\rho}{v}\right)^{q-p+1}\nabla v 
		\quad 
		{\rm in}\ O_\rho. 
	\end{equation}
	Since $\partial \phi_\rho/\partial \nu < 0$ on $\partial O_{\rho}$ and $q-p+1>0$ by \cite[Lemma~1.6]{BT_Picone}, 	
	we see that $\int_{O_\rho}\phi_\rho^{q-p}\,dx<\infty$.
	Combining this fact with $\min_{\overline{O_\rho}} v>0$, we conclude from \eqref{eq:picone-gen1} that $\phi_\rho^{q-p+1}v^{p-q} \in W_0^{1,1}(O_\rho)$.
	Approximating now $\phi_\rho^{q-p+1}v^{p-q}$ by functions from $C_0^\infty(O_\rho)$ in a standard way, we obtain
	\begin{equation}\label{eq:thm5-7}
		\int_{O_\rho} |\nabla \phi_\rho|^{p-2}\nabla \phi_\rho\,\nabla 
		\left(\phi_\rho^{q-p+1}v^{p-q}\right)dx
		=
		\lambda_1(m;O_\rho) \int_{O_\rho} mv^{p-q}\phi_\rho^q\,dx.
	\end{equation}
	Consequently, \eqref{eq:thm5-4}, 
	\eqref{eq:thm5-5}, \eqref{eq:thm5-6}, \eqref{eq:thm5-7} lead to the following contradiction: for any sufficiently small $\rho>0$, we have
	$$
	0>(\lambda_1(m;O_\rho)-\lambda)\int_{O_\rho} mv^{p-q}\phi_\rho^q\,dx
	\ge \eta \int_{O_\rho} a\phi_\rho^q\,dx\ge 0, 
	$$
	where the first inequality holds since
	$v, \phi_\rho>0$ in $O_\rho$ and $m=m_+$ is nonzero in $O_\rho$. 
\end{proof*}

\subsection{Additional properties}

In this final subsection, we prove additional qualitative properties of solutions of the problem \eqref{eq:Psub} stated in Section~\ref{subsec:additional}.
\begin{proof*}{Proposition~\ref{prop:non-exists-01}}\label{page:prop:non-exists-01:proof}
	Suppose, by contradiction, that there exists a solution $u$ of \eqref{eq:Psub} with some $0 \leq \lambda\leq \lambda_1(m)$ and either 
	$-\eta^*_\lambda(-a) < \eta \leq 0$ 
	or 
	$0 \leq \eta < \eta^*_\lambda(a)$ satisfying $u_- \not\equiv 0$ in $\Omega$.
	Taking $-u_- \in \W(\Omega) \setminus \{0\}$ as a test function for \eqref{eq:Psub}, we get 
	$$
	\intO |\nabla u_-|^p \,dx 
	-
	\lambda \intO mu_-^p\,dx 
	- 
	\eta \intO a u_-^q\,dx 		
	+
	\intO fu_-\,dx = 0.
	$$
	If either $\eta \intO a u_-^q\,dx \leq 0$, or $\lambda < \lambda_1(m)$, or $\lambda = \lambda_1(m)$ and 	
	$u_- \neq t \varphi_{1}$ for any $t>0$, then we obtain a contradiction to Lemma~\ref{lem:eta}.
	Therefore, assume that $\eta \intO a u_-^q\,dx > 0$, $\lambda = \lambda_1(m)$, and there exists $t_0>0$ such that $u_- = t_0 \varphi_{1}$.
	We deduce from \eqref{eq:Psub} and \eqref{eq:EP} that 
	$f = \eta a (t_0 \varphi_{1})^{q-1}$.
	Since $f$ is nonnegative and nontrivial, we have either $\eta > 0$ and $a \geq 0$ a.e.\ in $\Omega$, or $\eta < 0$ and $a \leq 0$ a.e.\ in $\Omega$.
	In the first case, we get $0 < \eta < \eta^*_{\lambda_1(m)}(a)$ and $\intO a \varphi_{1}^q \,dx>0$. However, it contradicts the fact that $\eta^*_{\lambda_1(m)}(a)=0$, see Remark~\ref{rem:eta}.
	The same contradiction is obtained in the second case.
\end{proof*}

\begin{proof*}{Proposition~\ref{prop:non-exists-0}}\label{page:prop:non-exists-0:proof}
	First, we suppose, by contradiction, that there exists a nonnegative solution $u$ of \eqref{eq:Psub} for some $\lambda\ge \lambda_1(m)$ and $\eta \geq 0$. 
	In view of the regularity assumptions \ref{WM}, \ref{WA}, \ref{WF1} and the nonnegativity of $a, f$, we can apply the weak Harnack inequality given by \cite[Theorem~7.1.2 and a subsequent remark]{PS} when $p \leq N$ (see also \cite[Corollary~7.1.3]{PS}) and \cite[Theorem~7.4.1]{PS} when $p > N$ to
	deduce that $u>0$ in $\Omega$. 
	Lemma~\ref{lem:testfunc} guarantees that $\varphi_1^p/(u+\varepsilon)^{p-1} \in \W(\Omega)$ for any $\varepsilon>0$, i.e., it can be used as a test function for \eqref{eq:Psub}. 
	Applying the Picone inequality from  Lemma~\ref{lem:picone-weak}, we get
	\begin{align}
		\notag
		\lambda_1(m) \intO m 
		\varphi_1^p \, dx  
		&=
		\intO |\nabla \varphi_1|^{p} \,dx
		\geq
		\intO |\nabla u|^{p-2} \nabla u\nabla \left(\frac{\varphi_1^p}{(u+\varepsilon)^{p-1}}\right) dx\\
		&=
		\lambda \intO m \, \frac{u^{p-1}}{(u+\varepsilon)^{p-1}} \, \varphi_1^p \, dx
		+
		\eta \intO a \, \frac{u^{q-1}}{(u+\varepsilon)^{p-1}} \, \varphi_1^p \, dx
		+
		\intO \frac{f\varphi_1^p}{(u+\varepsilon)^{p-1}}\,dx\\
		\label{eq:prop310}
		&\geq 
		\lambda \intO m \, \frac{u^{p-1}}{(u+\varepsilon)^{p-1}} \,\varphi_1^p \, dx
		+
		\frac{1}{(\|u\|_\infty+\varepsilon)^{p-1}}\intO f\varphi_1^p\,dx,
	\end{align}
thanks to the assumptions $a,f \geq 0$ a.e.\ in $\Omega$ and $\eta \geq 0$.
Recalling that $u>0$ in $\Omega$, we use the dominated convergence theorem to obtain
	\begin{equation*}
		\intO m \,
		\frac{u^{p-1}}{(u+\varepsilon)^{p-1}} \,
		\varphi_1^p\, dx
		\to
		\intO m 
		\varphi_1^p\, dx
		> 0
		\quad \text{as}~ \varepsilon \searrow 0.
	\end{equation*}
	Noting that $\intO f\varphi_1^p\,dx>0$ in view of the assumptions $f \not\equiv 0$ and $f \geq 0$ a.e.\ in $\Omega$, we observe that the last term in \eqref{eq:prop310} is uniformly bounded with respect to $\varepsilon>0$ from below by a positive number.
	Therefore, we pass to the limit as $\varepsilon \searrow 0$ in \eqref{eq:prop310} and derive a contradiction to our assumption $\lambda \geq \lambda_1(m)$.
	
	Let us now cover the case $\hat{\eta}_\lambda \leq  \eta < 0$ provided $\lambda > \lambda_1(m)$.
	Suppose, by contradiction, that there exist $\lambda > \lambda_1(m)$ and a sequence $\eta_n \nearrow 0$ such that {\renewcommand{\Plef}{\mathcal{P};\lambda,\eta_n}\eqref{eq:Psub}} possesses a nonnegative solution $u_n$. 
	Let us show that $\{u_n\}$ converges in $\W(\Omega)$, up to a subsequence, to a nonnegative solution of {\renewcommand{\Plef}{\mathcal{P};\lambda,0}\eqref{eq:Psub}}, i.e., \eqref{eq:Pfred}. 
	Then the first part of the proof will yield a contradiction. 
	For this purpose, we first show that $\{u_n\}$ is bounded in $\W(\Omega)$.
	If we suppose, by contradiction, that $\|\nabla u_n\|_p \to \infty$, up to a subsequence, then Lemma~\ref{lem:auxil0} guarantees that $\lambda$ is an eigenvalue of the problem~\eqref{eq:EP} and $\{u_n/\|\nabla u\|_p\}$ converges in $\W(\Omega)$ to a nonnegative eigenfunction of \eqref{eq:EP}, up to a subsequence.
	However, it is known that $\lambda_1(m)$ is the unique positive eigenvalue of \eqref{eq:EP} which has a corresponding sign-constant eigenfunction, see, e.g., \cite[Theorem~3.2]{Cuesta}.
	This contradicts our assumption $\lambda>\lambda_1(m)$.
	Thanks to the boundedness of $\{u_n\}$ in $\W(\Omega)$,
	we deduce from Lemma~\ref{lem:convsol} that it converges in $\W(\Omega)$ to a nonnegative solution $u$ of \eqref{eq:Pfred}, up to a subsequence.
	But this is impossible in view of the first part of the proof which covers the case $\eta=0$. 
\end{proof*}

\section{Existence of solutions. Proof of Theorem~\ref{thm:existence}}\label{sec:existence}
\label{page:thm:existence:proof}

Recall from Section~\ref{subsec:existence} that solutions of \eqref{eq:Psub} are in one-to-one correspondence with critical points of the energy functional $\E \in C^1(\W(\Omega), \mathbb{R})$ defined as
\begin{equation*}\label{def:E} 
	\E(u) = \frac{1}{p}\, H_\lambda (u)
	-\frac{\eta}{q}\intO a|u|^q\,dx - \intO f u \,dx,
	\quad u \in \W(\Omega),
\end{equation*}
where
$$
H_\lambda(u)=\intO |\nabla u|^p\,dx-\lambda\intO m|u|^p\,dx. 
$$
In order to prove Theorem \ref{thm:existence}, we will show that $\E$ has a linking structure provided $\lambda \neq \lambda_k(m)$, 
$k \in \mathbb{N}$, where the sequence of eigenvalues $\{\lambda_k(m)\}$ of \eqref{eq:EP} is defined in Section~\ref{sec:eigen}.
The arguments are similar to those presented in \cite[Section 3.1]{BobkovTanaka2019}.
For reader's convenience, we give a sketch of the proof. 

Throughout this section, 
we always assume that \ref{WM}, \ref{WA}, \ref{WF1} are satisfied, and 
either of the assumptions \ref{thm:existence:1}, \ref{thm:existence:2}, \ref{thm:existence:3} of Theorem~\ref{thm:existence} holds. 
In this case, Lemma~\ref{lem:PS} guarantees that $\E$ satisfies the Palais--Smale condition. 

If $m_-$ is nontrivial, then $\lambda\in\sigma(-\Delta_p;m)$ 
if and only if $-\lambda\in\sigma(-\Delta_p;-m)$ 
(see Section~\ref{sec:eigen}), and hence it is sufficient to handle only the case $\lambda \geq 0$.
If $m_- = 0$ a.e.\ in $\Omega$, then all the subsequent results remain valid also for $\lambda < 0$.

\subsection{Linking structure} 
Taking any $\lambda \geq 0$, 
we consider the set
\begin{equation}\label{def:Y} 
	Y(\lambda;m)
	:=
	\left\{
	u\in\W(\Omega):\, \|\nabla u\|_p^p \geq \lambda \intO m|u|^p\,dx  
	\right\}.
\end{equation}
Recall from Section~\ref{sec:eigen} that $S^{k}$ stands for the unit sphere in $\mathbb{R}^{k+1}$.
Denoting $S^k_+ = \{x=(x_1,\dots,x_{k+1}) \in S^k:\, x_{k+1} \geq 0\}$, we have $\partial S^k_+ = S^{k-1}$.
\begin{lemma}\label{lem:link} 
	Let $k \in \mathbb{N}$.
	If $h\in C(S^k_+,\W(\Omega))$ and $h\big|_{S^{k-1}}$ is odd, then 
	$h(S^k_+)$ intersects with $Y(\lambda_{k+1}(m);m)$.
\end{lemma}
\begin{proof} 
	Let us take any $h$ as required. 
	If there exists $z \in S^k_+$ such that $\intO m |h(z)|^p \,dx \leq 0$, then $h(z) \in Y(\lambda_{k+1}(m);m)$.
	Therefore, assume that $\intO m |h(z)|^p \,dx > 0$ for any $z \in S^k_+$.
	Let us define a normalization $h'$ of $h$ as  $h'(z) = h(z)/ (\intO m |h(z)|^p \,dx)^{1/p}$ for $z \in S^k_+$.
	Thus, we have $h' \in C(S^k_+,S(m))$, where the subset $S(m)$ of $\W(\Omega)$ is defined by \eqref{S_m}.  
	Consider the odd extension $\tilde{h} \in C(S^k,S(m))$ of $h'$ defined as $\tilde{h}(z):=-h'(-z)$ for $z\in S^k\setminus S^k_+$. 
	That is, we have $\tilde{h}\in \mathscr{F}_{k+1}(m)$,
	where the set $\mathscr{F}_{k+1}(m)$ is defined by \eqref{F_n},  and so 
	the definition \eqref{lambda_n} of $\lambda_{k+1}(m)$ gives 
	$$
	\max_{z \in S^k_+} \|\nabla h'(z)\|_p^p
	=
	\max_{z \in S^k} \|\nabla \tilde{h}(z)\|_p^p \ge 
	\lambda_{k+1}(m).
	$$
	Consequently, there exists $z \in S^k_+$ such that $\|\nabla h'(z)\|_p^p \ge 
	\lambda_{k+1}(m)$, 
	which yields $h(z) \in Y(\lambda_{k+1}(m);m)$. 
\end{proof}

\begin{lemma}\label{lem:EY}
	Let $0 \leq \lambda < \xi$ and $\eta \in \mathbb{R}$.
	Then $\E$ is bounded from below and coercive on $Y(\xi;m)$. 
\end{lemma} 
\begin{proof} 
	Recalling our default assumption $1<q<p$, it is sufficient to justify the existence of a constant $C>0$ such that $H_\lambda(u)\ge C\|\nabla u\|_p^p$ 
	for all $u \in Y(\xi;m)$. 
	Take any $u\in Y(\xi;m)$. 
	If $\intO m |u|^p\,dx\le 0$, then the desired inequality is obvious with $C=1$. 
	If $\intO m |u|^p\,dx> 0$, then 
	we get 
	$$
	H_\lambda(u)\ge \left(1- \frac{\lambda}{\xi}\right) \|\nabla u\|_p^p. 
	$$
	Thus, taking $C=\min\{1,1-\lambda/\xi\}>0$, we obtain the claim.
\end{proof}

\subsection{Case \texorpdfstring{$\lambda\ge 0$}{lambda>=0} and 
	\texorpdfstring{$\lambda\not\in \{\lambda_k(m):\,k\in\mathbb{N}\}$}{lambda not an eigenvalue}}\label{subsec:nonresonant}

It is clear that if $0\le \lambda<\lambda_1(m)$, 
then $Y(\lambda_{1}(m);m) = \W(\Omega)$, and hence 
there exists a global minimizer of $\E$ which is a solution of \eqref{eq:Psub}. 
(If $m_- = 0$ a.e.\ in $\Omega$, then the same is true also for $\lambda \leq 0$.)
Assume that $\lambda > \lambda_1(m)$ and $\lambda\not\in \{\lambda_k(m):k\in\mathbb{N}\}$. 
(Observe that, hypothetically, $\lambda$ might still be an eigenvalue of \eqref{eq:EP}, in which case we recall that either the assumption \ref{thm:existence:2} or \ref{thm:existence:3} of Theorem~\ref{thm:existence} is satisfied.)
Since $\lambda_k(m)\to \infty$ as $k \to \infty$ (see \cite[Remark~2.1]{Cuesta}), there exists $k \in \mathbb{N}$ such that $\lambda_k(m)<\lambda<\lambda_{k+1}(m)$.
Define 
\begin{align*} 
	\omega 
	&:=
	\inf\left\{\E(u):\, u\in Y(\lambda_{k+1}(m);m)\right\}, 
	\\
	\Lambda 
	&:=
	\left\{h\in C (S_+^k,\W(\Omega)):\, \max_{z\in S^{k-1}}\E(h(z))
	\le \omega-1\ \text{and}\ h\big|_{S^{k-1}} \ \text{is\ odd}\,\right\}, 
	\\
	c 
	&:= 
	\inf_{h\in \Lambda}\max_{z\in S_+^k} \E(h(z)).
\end{align*}
According to Lemma \ref{lem:EY}, $\omega$ is bounded from below. 
If $\Lambda \neq \emptyset$, then Lemma~\ref{lem:link} yields $c\ge \omega$. 
Recalling that $\E$ satisfies the Palais--Smale condition by Lemma~\ref{lem:PS}, 
we use standard arguments based on the deformation lemma to deduce that $c$ is a critical level of $\E$, see, e.g., \cite[pp.~3023-3024]{Tanaka}.
Therefore, it remains to verify that $\Lambda\not=\emptyset$.

Let us take any $0<\varepsilon<(\lambda-\lambda_k(m))/2$. 
Thanks to the definition \eqref{lambda_n} of $\lambda_k(m)$, we can find
$h_0\in \mathscr{F}_k(m)$ such that 
$$
\max_{z\in S^{k-1}}\|\nabla h_0(z)\|_p^p<\lambda_k(m)+\varepsilon,
$$
and hence
\begin{equation}\label{eq:H<0}
	\max_{z\in S^{k-1}} H_\lambda(h_0(z))<
	\lambda_k(m)+\varepsilon-\lambda 
	<-\varepsilon.
\end{equation}
In view of the assumptions $1<q<p$ and \ref{WA}, \ref{WF1}, there exists $T_0>0$ such that
\begin{align}
	\max_{z\in S^{k-1}} \E(T h_0(z))
	<-\frac{T^p \varepsilon}{p}
	&+\frac{T^q}{q} |\eta|\|a\|_\gamma \max_{z\in S^{k-1}}
	\|h_0(z)\|_{q\gamma/(\gamma-1)}^q 
	\nonumber \\ 
	&+ T \|f\|_\gamma  \max_{z\in S^{k-1}}
	\|h_0(z)\|_{\gamma/(\gamma-1)}
	\le \omega -1
	\label{eq:E<g-1} 
\end{align}
for any $T \geq T_0$.
Here, we set $\gamma/(\gamma-1) = \infty$ if $\gamma=1$ (in the case $N<p$).
Using \cite[Theorem~4.1]{dug}, we can extend $h_0$ from $S^{k-1}$ to $S_+^k$. 
Thus, we see from \eqref{eq:E<g-1}  that $T h_0 \in \Lambda$, which then implies that $c$ is a critical value of $\E$.
\qed

\subsection{Case \texorpdfstring{$\lambda=\lambda_k(m)$}{lambda=lambda-k} under assumption \ref{thm:existence:2}}\label{subsec:resonant+}

We may assume that $k \in \mathbb{N}$ is such that $\lambda=\lambda_k(m)<\lambda_{k+1}(m)$.
Take any sequence $\lambda_n \searrow \lambda$ and assume, without loss of generality, that $\lambda_k(m)<\lambda_n<\lambda_{k+1}(m)$ for all $n$.
Similarly to Section~\ref{subsec:nonresonant} above, we define
\begin{align*} 
	\omega_n 
	&:=
	\inf\{\En(u):\, u\in Y(\lambda_{k+1}(m);m)\}, 
	\nonumber \\
	\Lambda_n 
	&:=
	\left\{h\in C (S_+^k,\W(\Omega)):\, \max_{z\in S^{k-1}}\En(h(z))
	\le \omega_n-1\ \text{and}\ h\big|_{S^{k-1}} \ \text{is\ odd}\,\right\}, 
	\nonumber \\
	c_n 
	&:= 
	\inf_{h\in \Lambda_n}\max_{z\in S_+^k} \En(h(z))
	\label{def:cn}
\end{align*}
for every $n$.
Arguing as in Section \ref{subsec:nonresonant}, we get $\Lambda_n \neq \emptyset$, and Lemma~\ref{lem:link} implies that $c_n\ge \omega_n$. 
Since $\{\lambda_n\}$ is decreasing, 
we have 
$H_{\lambda_n}(u)\ge H_{\lambda_1}(u)$ 
and hence $E_{\lambda_n,\eta}(u) \geq E_{\lambda_1,\eta}(u)$
for any
$u\in\W(\Omega)$ 
such that $\intO m |u|^p \,dx \geq 0$. 
On the other hand, for any
$u\in\W(\Omega)$ with $\intO m |u|^p \,dx < 0$ we have 
$E_{\lambda_n,\eta}(u) \geq E_{0,\eta}(u)$.
Therefore, we deduce that
\begin{equation}
	\label{eq:cn>gn}
	c_n \ge \omega_n \ge 
	\min 
	\left\{ 
	\inf_{Y(\lambda_{k+1}(m);m)}E_{\lambda_1,\eta}, 
	\inf_{\W(\Omega)}E_{0,\eta}
	\right\}
	>-\infty
	\quad 
	\text{for any}~ n,
\end{equation}
where the last inequality follows from Lemma \ref{lem:EY} and the coercivity of the functional $E_{0,\eta}$ on $\W(\Omega)$. 
Consequently, $\{c_n\}$ is bounded from below. 

Arguing as in {\cite[Section 3.2]{BobkovTanaka2019}}, 
for each $n$ we can find a function $u_n \in \W(\Omega)$ such that
\begin{equation}
	\label{eq:e<1/n}
	|\En(u_n)-c_n| < \frac{1}{n} 
	\quad \text{and}\quad 
	\|\En^\prime(u_n)\|_*<\frac{1}{n}.
\end{equation}
Let us show that 
$\{u_n\}$ is bounded in $\W(\Omega)$.
Indeed, if $\|\nabla u_n\|_p \to \infty$ along a subsequence, 
then, by standard arguments, the second inequality in \eqref{eq:e<1/n} implies that the normalized functions $v_n = u_n/\|\nabla u_n\|_p$ converge in $\W(\Omega)$ to some $v_0\in ES(\lambda;m)\setminus\{0\}$, up to a subsequence.
Recalling that $\{c_n\}$ is bounded from below and passing to the limit in 
\begin{align*} 
	\frac{p(c_n-1/n)}{\|\nabla u_n\|_p^q} 
	-\frac{1}{n\|\nabla u_n\|_p^{q-1}}
	&\le \frac{1}{\|\nabla u_n\|_p^q}
	\left(\,p\En(u_n) - \left<\En^\prime (u_n),u_n\right> \right) 
	\\
	&=-\left(\frac{p}{q}-1\right)\eta\intO a|v_n|^q\,dx 
	-\frac{p-1}{\|\nabla u_n\|_p^{q-1}}\,\intO f v_n \,dx, 
\end{align*}
we deduce that $\eta\intO a|v_0|^q\,dx\le 0$.
However, this contradicts the imposed assumption~\ref{thm:existence:2}, which implies that $\{u_n\}$ is bounded in $\W(\Omega)$. 
Thanks to \ref{WM}, we have
\begin{align*}
	\|\E^\prime(u_n)\|_*
	\le \|\E^\prime(u_n)-\En^\prime(u_n)\|_*
	+\|\En^\prime(u_n)\|_*
	\leq 
	C (\lambda_n-\lambda)\|\nabla u_n\|_p^{p-1}
	+\frac{1}{n},
\end{align*}
where $C>0$ is independent of $u_n$ (cf.\ \eqref{eq:Enorm} in Lemma~\ref{lem:convsol}).
Therefore, 
we see that $\{u_n\}$ is a bounded Palais--Smale sequence for $\E$.
Applying Lemma~\ref{lem:PS0}, we conclude that $\{u_n\}$ converges in $\W(\Omega)$ to a critical point of $\E$, up to a subsequence.

\subsection{Case \texorpdfstring{$\lambda=\lambda_k(m)$}{lambda=lambda-k} under assumption~\ref{thm:existence:3}}

We may assume that $k \in \mathbb{N}$ is such that $\lambda_k(m)<\lambda=\lambda_{k+1}(m)$.
Consider any sequence $\lambda_n \nearrow \lambda$ and assume, without loss of generality, that $\lambda_k(m)<\lambda_n < \lambda_{k+1}(m)$ for all $n$.
Let us define $\omega_n$, $\Lambda_n$, and $c_n$ as in Section~\ref{subsec:resonant+} above. 

As in Section~\ref{subsec:nonresonant}, we can find $h_0 \in C(S^k_+,\W(\Omega))$ such that $h_0\big|_{S^{k-1}}$ is odd, $\intO m |h_0(z)|^p \,dx = 1$ for all $z \in S^{k-1}$, and $\max_{z \in S^{k-1}} E_{\lambda_1,\eta}(T h_0(z)) \to -\infty$ as $T \to \infty$.
Then, arguing in a similar way as in \cite[Section~3.3]{BobkovTanaka2019}, 
we can prove that $\{c_n\}$ is bounded from above and 
for any $n$ one can find $u_n \in \W(\Omega)$ satisfying 
$$
|\En(u_n)-c_n| < \frac{1}{n} 
\quad \text{and}\quad 
\|\En^\prime(u_n)\|_* < \frac{1}{n}.	
$$
Then, as in Section~\ref{subsec:resonant+}, the inequality 
\begin{align*} 
	\frac{p(c_n+1/n)}{\|\nabla u_n\|_p^q} 
	+\frac{1}{n\|\nabla u_n\|_p^{q-1}}
	&\ge \frac{1}{\|\nabla u_n\|_p^q}
	\left(\,p\En(u_n) - \left<\En^\prime (u_n),u_n\right> \right) 
	\\
	&=-\left(\frac{p}{q}-1\right)\eta\,\intO a|v_n|^q\,dx 
	-\frac{p-1}{\|\nabla u_n\|_p^{q-1}}\,\intO f v_n \,dx,
\end{align*}
where $v_n = u_n/\|\nabla u_n\|_p$, 
in combination with the imposed assumption~\ref{thm:existence:3} implies the boundedness of $\{u_n\}$ in $\W(\Omega)$, which yields the existence of a critical point of $\E$.

\appendix
\section{Regularity}\label{sec:regularity}

We start with an $L^\infty(\Omega)$-bound for solutions of the problem \eqref{eq:Psub}.
In view of the Morrey lemma, it will be sufficient to investigate only the case $N \geq p$. 
The proof uses the classical bootstrap argument and we present it sketchily for completeness.
\begin{proposition}\label{prop:bdd} 
	Let $N \geq p>1$.
	Assume that 
	$m, a, f \in L^\gamma(\Omega)$ for some $\gamma>N/p$, and $M_1 > 0$ is any constant such that $\|m\|_\gamma, \|a\|_\gamma, \|f\|_\gamma \leq M_1$. 
	Let $r$ be such that $p\gamma^\prime < r < p^*$,  
	where $\gamma^\prime = \gamma/(\gamma-1)$ and $p^*$ is defined in \eqref{eq:pstar}. 
	Assume that $|\lambda|, |\eta| \leq M_2$ for some $M_2 > 0$.
	Then there exists 
	$C=C(|\Omega|,M_1,M_2,p,q,\gamma,r)>0$ 
	such that any solution $u$ of \eqref{eq:Psub} 
	satisfies
	\begin{equation}\label{eq:linfty}
		\|u\|_\infty \le C \left(\,1+\|u\|_{r}\right). 
	\end{equation}
\end{proposition} 
\begin{proof} 
	Let $C_*>0$ be the best constant of the embedding $\W(\Omega) \hookrightarrow L^r(\Omega)$, and set
	$M_0 =\max\{1,|\Omega|\}$. 
	Let $u$ be any solution of \eqref{eq:Psub} and denote, for brevity, $v = u_+$.
	For any $l>0$ and $M>0$, we denote  
	$v_M = \min\{v,M\}$ and take $v_M^{l+1} \in \W(\Omega)$ as a test function for \eqref{eq:Psub}.
	Concerning the left-hand side of \eqref{eq:Psub}, we have 
	\begin{align} 
		\intO |\nabla u|^{p-2}\nabla u\nabla (v_M^{l+1})\,dx 
		&=(l+1)\intO v_M^l|\nabla v_M|^p\,dx
		=(l+1)\left(\frac{p}{p+l}\right)^p\,\intO |\nabla (v_M^{1+l/p})|^p\,dx
		\\
		\label{eq:infty1}
		&\ge (l+1)\left(\frac{p}{C_*(p+l)}\right)^p
		\|v_M^{1+l/p}\|_{r}^p
		\geq
		\frac{1}{C_*^{p}(p+l)^p} \|v_M\|_{(p+l)r/p}^{p+l}. 
	\end{align}
	On the other hand, temporarily
	assuming that $v\in L^{(p+l)\gamma^\prime}(\Omega)$,
	we apply the H\"older inequality to estimate the right-hand side of \eqref{eq:Psub} from above by the following expressions:
	\begin{align} 
		\notag
		&|\lambda|\|m\|_{\gamma}\|v\|_{(p+l)\gamma^\prime}^{p+l} 
		+|\eta|\|a\|_{\gamma}\|v\|_{(q+l)\gamma^\prime}^{q+l} 
		+\|f\|_{\gamma} \|v\|_{(1+l)\gamma^\prime}^{1+l} 
		\\ 
		\label{eq:linf1}
		&\le 
		M_1
		\Big(M_2 
		+ 
		M_2 M_0^\frac{(p-q)}{(p+l)\gamma^\prime}
		+
		M_0^\frac{(p-1)}{(p+l)\gamma^\prime}
		\Big) \max\{1,\|v\|_{(p+l)\gamma^\prime}^{p+l}\}
		\\
		\label{eq:infty2}
		&\leq
		M_1
		\Big(M_2 
		+ 
		M_2 M_0^\frac{(p-q)}{p\gamma^\prime}
		+
		M_0^\frac{(p-1)}{p\gamma^\prime}
		\Big) \max\{1,\|v\|_{(p+l)\gamma^\prime}^{p+l}\}.
	\end{align}
	Consequently, if $v\in L^{(p+l)\gamma^\prime}(\Omega)$ for some $l>0$, 
	then we let $M\to \infty$ and deduce from \eqref{eq:infty1} and \eqref{eq:infty2} that
	$v\in L^{(p+l)r/p}(\Omega)$ and 
	\begin{align}
		\|v\|_{(p+l)r/p} 
		\le 
		\left(C(p+l)^p\right)^{1/(p+l)}\,
		\max\{1,\|v\|_{(p+l)\gamma^\prime}\}, 
		\label{eq:bdd-2}
	\end{align} 
	where $C\ge 1$ is a constant independent of $l>0$ and $v$.
	Now we define a sequence $\{l_m\}$ as follows: 
	$$
	(p+l_0)\gamma^\prime=r\quad {\rm and}\quad 
	l_{m+1}=\frac{p+l_m}{P}-p, \quad {\rm where}\quad 
	P:=\frac{\,\gamma^\prime p\,}{r}<1. 
	$$
	In particular, we do have $v\in L^{(p+l_0)\gamma^\prime}(\Omega)$.
	Denoting 
	$d_m = (C(p+l_m)^p)^{1/(p+l_m)}$, 
	we infer from \eqref{eq:bdd-2} that 
	\begin{equation}\label{eq:bdd-3} 
		\|v\|_{(p+l_{m+1})\gamma^\prime}\le d_m \max\{1,\|v\|_{(p+l_m)\gamma^\prime}\} 
		\le  \max\{1,\|v\|_{(p+l_0)\gamma^\prime}\}\prod_{k=0}^m d_k
	\end{equation} 
	for every $m$. 
	Let us show that $\prod_{k=0}^\infty d_k$ is finite. 
	We have 
	\begin{equation}\label{eq:bdd-4}
		\log \prod_{k=0}^\infty d_k=
		\log C\,\sum_{k=0}^\infty \frac{1}{p+l_k}+
		p\,\sum_{k=0}^\infty \frac{\log (p+l_k)}{p+l_k}. 
	\end{equation}
	Noting that $P=(p+l_k)/(p+l_{k+1})<1$, 
	we deduce that $l_k \to \infty$ and 
	get 
	$$
	\frac{\log (p+l_{k+1})}{p+l_{k+1}}\, 
	\frac{p+l_k}{\log (p+l_k)}
	=P\,\left(1-\frac{\log P}{\log (p+l_k)} \right)
	\to P<1 
	\quad \text{as}~k\to\infty.
	$$
	Hence, 
	the series in \eqref{eq:bdd-4} are convergent, which completes the proof of the boundedness of $v$ by letting $m\to\infty$ in \eqref{eq:bdd-3}. 
	Repeating the same procedure with $v=-u_-$, we obtain the boundedness of $u_-$ and hence of $u$ in the form \eqref{eq:linfty}.
\end{proof}

After we established the boundedness of solutions of \eqref{eq:Psub}, we can consider the whole right-hand side of \eqref{eq:Psub} as a function which maps $\Omega$ to $\mathbb{R}$ and investigate the regularity of solutions of the corresponding Poisson problem in order to get the regularity of solutions of \eqref{eq:Psub}.
Namely, we consider the problem
\begin{equation}\label{eq:P0}
	\left\{
	\begin{aligned}
		-\Delta_p u &= g(x) 
		&&\text{in}\ \Omega, \\
		u&=0 &&\text{on}\ \partial \Omega.
	\end{aligned}
	\right.
\end{equation}
The following result on the local H\"older regularity of the (unique) solution of \eqref{eq:P0} is well known (see, e.g., \cite[Theorem~7.3.1]{PS}) and we omit the proof.
\begin{proposition}\label{prop:C0-reg}
	Let $\|g\|_\gamma \leq M$ for some $M >0$, where $\gamma>N/p$ if $N\ge p$ and $\gamma=1$ if $N<p$. 
	Then the solution $u \in \W(\Omega)$ of \eqref{eq:P0} satisfies $u \in C(\Omega) \cap L^\infty(\Omega)$.
	Moreover, there exists $\beta = \beta(p,N,\gamma) \in (0,1)$ such that 
	for any compact subset $K \subset \Omega$ there exists $C = C(\Omega,K,M,p,\gamma)>0$ such that  $\|u\|_{C^{0,\beta}(K)} \leq C$.
\end{proposition}

Let us discuss a higher regularity of solutions of \eqref{eq:P0} under a higher integrability assumption on the source function $g$. 
The proof of the following result is inspired by \cite[Proposition~2.1]{KZ} (see also \cite[Proposition~2.1]{APO} and compare with \cite[Corollary]{dib}). 
We expand the approach of \cite[Proposition~2.1]{KZ} in order to provide more explicit dependence 
of the regularity of solutions of \eqref{eq:P0} on $g$, which is necessary for the proofs of our main results formulated under the assumptions \ref{O}, \ref{M}, \ref{A}, \ref{F1}.
\begin{proposition}\label{prop:C1-reg}
	Let $\Omega$ satisfy \ref{O}.
	Let $\|g\|_\gamma \leq M$ for some $M >0$ and $\gamma>N$.
	Then there exist $\beta = \beta(M,p,N,\gamma) \in (0,1)$ and $C = C(\Omega,M,p,\gamma)>0$ such that the solution $u \in \W(\Omega)$ of \eqref{eq:P0} satisfies 
	$u \in C^{1,\beta}(\overline{\Omega})$ and
	$\|u\|_{C^{1,\beta}(\overline{\Omega})} 
	\leq C$.
\end{proposition}
\begin{proof}
	Throughout the proof, we denote by $C>0$ a universal constant, for convenience.
	First, we consider the following problem for the linear Laplace operator:
	\begin{equation}\label{eq:P02}
		\left\{
		\begin{aligned}
			-\Delta v &= g(x) 
			&&\text{in}\ \Omega, \\
			v&=0 &&\text{on}\ \partial \Omega,
		\end{aligned}
		\right.
	\end{equation}
	where the function $g$ is the same as in \eqref{eq:P0}.
	Since $\Omega$ is of class $C^{1,1}$
	and $g \in L^\gamma(\Omega)$ with $\gamma>N$, the problem \eqref{eq:P02} has a unique solution $v \in W^{2,\gamma}(\Omega)$, see, e.g., \cite[Theorem~9.15]{GT} (in fact, here $\gamma>1$ is enough).
	Moreover, this solution has the following property, see, e.g., \cite[Lemma~9.17]{GT}:
	$$
	\|v\|_{W^{2,\gamma}(\Omega)} \leq C \|g\|_\gamma,
	$$
	where $C$ does not depend on $v$ and $g$. 
	Thanks to the regularity of $\Omega$  and the assumption $\gamma>N$, the embedding $W^{2,\gamma}(\Omega) \hookrightarrow C^{1,\kappa}(\overline{\Omega})$ is continuous with $\kappa = 1-\frac{N}{\gamma} \in (0,1)$, see, e.g., \cite[Theorem~7.26]{GT} (in fact, here $C^{0,1}$-regularity of $\Omega$ is enough).
	Consequently,
	\begin{equation}\label{eq:c1g<W}
		\|v\|_{C^{1,\kappa}(\overline{\Omega})}
		\leq
		C 
		\|v\|_{W^{2,\gamma}(\Omega)}
		\leq 
		C \|g\|_\gamma \leq C M,
	\end{equation}
	where $C$ is independent of $v$.
	Recall, for convenience, that
	$$
	\|v\|_{C^{1,\kappa}(\overline{\Omega})}
	:=
	\sup_{x \in \Omega} |v(x)|
	+
	\max_{i=1,\dots,N}\sup_{x \in \Omega} |v'_{x_i}(x)|
	+
	\max_{i=1,\dots,N}\sup_{x,y \in \Omega,~x \neq y} \frac{|v'_{x_i}(x)-v'_{x_i}(y)|}{|x-y|^\kappa}.
	$$
	Denoting $V(x) = \nabla v(x)$, we have 
	\begin{equation}\label{eq:vreg}
	V \in C^{0,\kappa}(\overline{\Omega}; \mathbb{R}^N).
	\end{equation}
	Subtracting \eqref{eq:P02} from \eqref{eq:P0}, we see that the solution $u$ of \eqref{eq:P0} weakly solves the problem
	\begin{equation}\label{eq:Pdiff}
		\left\{
		\begin{aligned}
			-\text{div}\left(|\nabla u|^{p-2} \nabla u - V(x)\right) &= 0
			&&\text{in}\ \Omega, \\
			u&=0 &&\text{on}\ \partial \Omega.
		\end{aligned}
		\right.
	\end{equation}
Let us show that the regularity result \cite[Theorem~1]{Lieberman} is applicable to \eqref{eq:Pdiff}.
	Denote $A(x,z) = |z|^{p-2} z - V(x)$ and $a^{ij}(z) = \frac{\partial A^i(x,z)}{\partial z_j}$, $z \in \mathbb{R}^N$.
	The matrix $(a^{ij}(z))$ is a symmetric $N \times N$-matrix corresponding to the linearization of the $p$-Laplacian and we have
	\begin{equation}\label{eq:martxA}
		(a^{ij}(z)) = |z|^{p-2} \left(I + (p-2) \frac{z \otimes z}{|z|^2}\right),
		\quad z \in \mathbb{R}^N \setminus \{0\},
	\end{equation}
	where $z \otimes z := (z_i z_j)$ is a matrix.
	We set $(a^{ij}(0))$ to be a zero matrix.
	It is not hard to see that
	\begin{equation}\label{eq:abb}
		\min\{1,p-1\}|z|^{p-2} |\xi|^2 \leqslant 
		\sum_{i,j=1}^N a^{ij}(z)\xi_i \xi_j \leqslant 
		\max\{1,p-1\} |z|^{p-2} |\xi|^2
	\end{equation}
	for any $z,\xi \in \mathbb{R}^N$, see, e.g., \cite[Section~5.1]{takac-lec2}.
	Thanks to \eqref{eq:vreg}, we have the following estimate for all $x,y \in \overline{\Omega}$ and $z \in \mathbb{R}^N$:
	$$
	|A(x,z)-A(y,z)|
	=
	|V(x)-V(y)|
	\leq
	C |x-y|^{\kappa},
	$$
	where $C$ depends on $M$ but does not depend on $v$, $x$, $y$, $z$.
	We also mention that since $\Omega$ satisfies \ref{O}, $\Omega$ automatically belongs to the class  $C^{1,\kappa}$.
	
	Finally, we recall that the solution $u$ of \eqref{eq:P0} is bounded. 
	More precisely, in the case $N \geq p$, Proposition \ref{prop:bdd} gives the bound
	\begin{equation}\label{eq:linfty1}
		\|u\|_\infty \le C\,(\,1+\|u\|_{r}), 
	\end{equation}
	where $p\gamma^\prime < r < p^*$ and $C$ does not depend on $u$, and a similar bound holds in the case $N<p$ due to the Morrey lemma. 
	Notice that since $\gamma>N$, we have $\gamma' < p^*$.
	Since $u$ satisfies $\intO |\nabla u|^p \, dx = \intO g u \,dx$, we use the H\"older inequality and the continuity of the embedding $\W(\Omega) \hookrightarrow L^{\gamma'}(\Omega)$ to deduce that
	$$
	\|\nabla u\|_p^p \leq \|g\|_\gamma \|u\|_{\gamma'} \leq C \|g\|_\gamma \|\nabla u\|_{p},
	$$
	which yields 
	$$
	\|u\|_{r} \leq C \|\nabla u\|_p \leq C \|g\|_\gamma^\frac{1}{p-1} \leq C M^\frac{1}{p-1},
	$$
	where $C$ does not depend on $M$ and $u$.
	Combining this estimate with \eqref{eq:linfty1}, we finally arrive at the bound
	$\|u\|_\infty \leq C$, where $C$ is independent of $u$.
	
	Thus, all the requirements of \cite[Theorem~1]{Lieberman} are satisfied, which guarantees that $u \in C^{1,\beta}(\overline{\Omega})$, where $\beta = \beta(M,p,N,\gamma) \in (0,1)$ and
	$\|u\|_{C^{1,\beta}(\overline{\Omega})} \leq C(\Omega,M,p,\gamma)$.
\end{proof}

\section{Weak form of the Picone inequality}\label{sec:picone}

In the proofs of Theorem~\ref{thm3} and Proposition~\ref{prop:thm1-w} (and hence of Theorem~\ref{thm1-w}~\ref{thm1-w:2}), we need to employ a version of the standard Picone inequality \cite[Theorem~1.1]{Alleg} applicable to purely Sobolev functions, i.e., when no a priori information on the a.e.-differentiability is available.
We start with the following auxiliary results.
\begin{lemma}\label{lem:testfunc}
	Let $\varphi \in W^{1,p}(\Omega) \cap L^\infty(\Omega)$, $u \in W^{1,p}(\Omega)$, and $\varepsilon>0$.
	Then $|\varphi|^p/(|u|+\varepsilon)^{p-1} \in W^{1,p}(\Omega) \cap L^\infty(\Omega)$ 
	and its weak gradient is expressed as follows:
	\begin{equation}\label{eq:piconeeq}
		\nabla \left(\frac{|\varphi|^p}{(|u|+\varepsilon)^{p-1}}\right)
		=
		p \frac{|\varphi|^{p-2}\varphi}{(|u|+\varepsilon)^{p-1}} \nabla \varphi
		-
		(p-1) \frac{|\varphi|^{p}}{(|u|+\varepsilon)^{p}} \left(\nabla u_+ + \nabla u_- \right).
	\end{equation}
	If, in addition, $\varphi \in W^{1,p}_0(\Omega)$, then $|\varphi|^p/(|u|+\varepsilon)^{p-1} \in W^{1,p}_0(\Omega)$.
\end{lemma}
\begin{proof}
	The proof is based on classical arguments, so we will be sketchy.
	First, we observe that $|\varphi|^p \in W^{1,p}(\Omega) \cap L^\infty(\Omega)$.
	Indeed, since $\varphi \in L^\infty(\Omega)$, we can find a function $G \in C^1(\mathbb{R})$ such that $G(s)=|s|^p$ for $s \in [-\|\varphi\|_\infty,\|\varphi\|_\infty]$ and $|G'(s)| \leq M$ for all $s \in \mathbb{R}$ and some uniform constant $M>0$. 
	Then \cite[Theorem~1.18]{HKM} ensures that $G(\varphi) \equiv |\varphi|^p \in W^{1,p}(\Omega)$ and its weak gradient is calculated according to the classical rules.
	Clearly, we also have $|\varphi|^p \in L^\infty(\Omega)$.
	
	It can be shown in a similar way that $1/(|u|+\varepsilon)^{p-1} \in W^{1,p}(\Omega) \cap L^\infty(\Omega)$. 
	Indeed, we can find a function $H \in C^1(\mathbb{R})$ such that $H(s) = 1/s^{p-1}$ for $s \in [\varepsilon,\infty)$ and $|H'(s)| \leq M$ for all $s \in \mathbb{R}$ and some uniform constant $M>0$. 
	Since $\Omega$ is bounded, we have $|u|+\varepsilon \in W^{1,p}(\Omega)$. 
	Hence, we deduce from \cite[Theorem~1.18]{HKM} that 
	$H(|u|+\varepsilon) \equiv 1/(|u|+\varepsilon)^{p-1} \in W^{1,p}(\Omega)$, and its weak gradient can be expanded by the classical rules. 
	Since $1/(|u|+\varepsilon)^{p-1} \leq 1/\varepsilon^{p-1}$, we conclude that $1/(|u|+\varepsilon)^{p-1} \in L^\infty(\Omega)$. 
	
	Applying now \cite[Theorem~1.24~(i)]{HKM}, we see that $|\varphi|^p/(|u|+\varepsilon)^{p-1} \in W^{1,p}(\Omega) \cap L^\infty(\Omega)$ and its	weak gradient is given by the expression \eqref{eq:piconeeq}.
	If we additionally assume that $\varphi \in W^{1,p}_0(\Omega)$, then, by a simple amendment of the proof of \cite[Theorem~1.18]{HKM}, we have $|\varphi|^p \in \W(\Omega)$, and hence $|\varphi|^p/(|u|+\varepsilon)^{p-1} \in \W(\Omega)$ by \cite[Theorem~1.24~(ii)]{HKM}.
\end{proof}

Under stronger requirements on the functions $\varphi$ and $u$, we can omit $\varepsilon$ in Lemma~\ref{lem:testfunc}. 
\begin{lemma}\label{lem:testfunc2}
	Let $\varphi \in W^{1,p}(\Omega) \cap L^\infty(\Omega)$ and $u \in W^{1,p}(\Omega)$
	be such that $K:=\mathrm{supp}\, \varphi \subset \Omega$ and $\mathrm{ess\,inf}_{K}\,|u| > 0$.
	Then $|\varphi|^p/|u|^{p-1} \in \W(\Omega) \cap L^\infty(\Omega)$ 
	and its weak gradient is expressed as in \eqref{eq:piconeeq} with $\varepsilon=0$.
\end{lemma}
\begin{proof}
	Arguments are similar to those from the proof of Lemma~\ref{lem:testfunc} and hence we omit details.
\end{proof}

In view of the expression \eqref{eq:piconeeq}, one can argue exactly as in the proof of \cite[Theorem~1.1]{Alleg} to obtain the following weak version of the Picone inequality, see also \cite[Section~2]{AP} and \cite[Section~3.2]{takac-lec2} for related results.
\begin{lemma}\label{lem:picone-weak}
	Let $\varphi \in W^{1,p}(\Omega) \cap L^\infty(\Omega)$ and $u \in W^{1,p}(\Omega)$
	be such that $u \geq 0$ a.e.\ in $\Omega$.
	Let $\varepsilon>0$.
	Then the following inequality holds:
	\begin{equation}\label{eq:weak-picone}
		\intO |\nabla u|^{p-2} \nabla u\nabla \left(\frac{|\varphi|^p}{(u+\varepsilon)^{p-1}}\right) dx
		\leq
		\intO |\nabla \varphi|^{p} \,dx.
	\end{equation}
	If, in addition, $K:=\mathrm{supp}\, \varphi \subset \Omega$ and $\mathrm{ess\,inf}_{K}\,u > 0$, then \eqref{eq:weak-picone} holds with $\varepsilon=0$.
\end{lemma}

\bigskip
\textbf{Acknowledgments.}
The authors are thankful to Prof.~\textsc{Ky Ho} for acquainting them with the works~\cite{APO,KZ} and for stimulating discussions.
V.~Bobkov was supported by RSF Grant Number 22-21-00580, \url{https://rscf.ru/en/project/22-21-00580/}.
M.~Tanaka was supported by JSPS KAKENHI Grant Number JP 19K03591. 


\end{document}